\documentclass[12pt]{amsart}

\usepackage{epsfig}
\usepackage{tikz,tikz-cd}
\usepackage{amssymb, amsmath}
\usepackage{amsfonts,graphics,epsfig}
\usepackage{mathrsfs}
\usepackage{pb-diagram, pb-xy}
\usepackage[all]{xy}
\usepackage{comment}
\usepackage{tikz}
\usepackage{setspace}
\usepackage[pagewise]{lineno}
\usepackage[left=1.5in,right=1.5in,top=1.5in,bottom=1.5in]{geometry}
\usepackage{color}   
\usepackage{hyperref}
\hypersetup{
    colorlinks=true, 
    linktoc=all,     
    linkcolor=red,
        citecolor=blue,  
}
\usepackage{tcolorbox}
\usepackage{enumitem}
\usepackage{soul}

\DeclareMathOperator{\Ex}{\mathrm{Ex}}
\def\mult{\operatorname{\rm mult}}

\def\mcF{\mathcal F}
\def\mcH{\mathcal H}
\def\M{\mathbf M}
\def\N{\mathbf N}
\def\A{\mathbf A}
\def\C{\mathbf C}
\def\H{\mathbf H}

\newtheorem{theorem}{Theorem}[section]
\newtheorem{lemma}[theorem]{Lemma}
\newtheorem{proposition}[theorem]{Proposition}
\newtheorem{corollary}[theorem]{Corollary}
\newtheorem{conjecture}[theorem]{Conjecture}

\theoremstyle{remark}
\newtheorem{remark}[theorem]{Remark}
\theoremstyle{definition}
\newtheorem{definition}[theorem]{Definition}
\theoremstyle{definition}

\numberwithin{equation}{section}
\theoremstyle{remark}
\newtheorem{claim}[theorem]{Claim}

\theoremstyle{definition}

\newcommand{\mbR}{\mathbb{R}}

\newcommand{\mbQ}{\mathbb{Q}}

\newcommand{\<}{\leq}
\def\>{\geq}

\def\vphi{\varphi}

\def\subset{\subseteq}

\newcommand{\num}{\equiv}

\def\mcF{\mathcal{F}}
\def\mcG{\mathcal{G}}
\def\mcH{\mathcal{H}}

\def\bfM{\mathbf{M}}

\def\T{\mathbf{\Theta}}
\def\H{\mathbf{H}}

\def\im{\operatorname{Im}}
\def\Supp{\operatorname{Supp}}
\def\dim{\operatorname{dim}}

\def\Ex{\operatorname{Ex}}

\def\NE{\overline{\operatorname{NE}}}

\def\hor{\operatorname{\hor}}
\def\ver{\operatorname{\ver}}

\def\mult{\operatorname{mult}}

\def\sing{\operatorname{\textsubscript{\rm sing}}}

\def\ver{\operatorname{\textsubscript{\rm ver}}}
\def\hor{\operatorname{\textsubscript{\rm hor}}}

\def\red{\operatorname{\textsubscript{\rm red}}}

\def\mbf{\mathbf}

\author{Priyankur Chaudhuri}
\address{School of Mathematics, Tata Institute of Fundamental Research, Homi Bhabha Road, Colaba,
Mumbai 400005. \newline Current address: Dipartimento di Matematica F. Enriques, Universita degli Studi di Milano, Via Cesare Saldini `
50, 20133 Milano, Italy.}

\email{pkurisibang@gmail.com}
\author{Omprokash Das}
\address{School of Mathematics, Tata Institute of Fundamental Research, Homi Bhabha Road, Colaba,
Mumbai 400005.}
\email{omprokash@gmail.com}

\title[Foliated Base-point free Theorem]{A basepoint free theorem for algebraically integrable foliations}

\begin{document}
\maketitle

\begin{abstract}
We show that if $\mathcal{F}$ is an algebraically integrable foliation on a $\mathbb{Q}$-factorial normal projective variety $X$, $ A, B \geq 0$ are $\mathbb{Q}$-divisors on $X$ with $A$ ample such that  $(\mathcal{F}, B)$ is foliated dlt, $(X,B)$ is klt and $K_{\mathcal{F}}+ A+B$ is nef, then $K_{\mathcal{F}}+A+B$ is semiample. We also provide some applications of this result such as contraction theorem for F-dlt pairs and a special case of the b-semiampleness conjecture.
\end{abstract}

\section{Introduction}
The study of birational geometry of algebraically integrable foliations was initiated by Ambro, Cascini, Shokurov and Spicer in \cite{ACSS}, where the authors used foliated MMP to study the positivity of the moduli part of lc-trivial (and more general) fibrations. They observe that the moduli divisor of a `nice' fibration can be related to a log canonical divisor of the foliation induced by it. In case the fibration is lc-trivial, the moduli b-divisor is known to be b-nef and b-good (\cite{Amb}, \cite{Hu}) and conjectured to be b-semiample \cite{PSh}. Hence abundance-type theorems for foliated log canonical pairs are expected to have applications to the b-semiampleness conjecture. \\

In \cite{CS3}, Cascini and Spicer stated a version of the following conjecture:

\begin{conjecture}\cite[Conjecture 4.1]{CS3}\label{conj}
Let $(\mcF, B)$ be an F-dlt pair (see Definition \ref{dlt}) on a $\mathbb{Q}$-factorial normal projective variety $X$, $\mcF$ is algebraically integrable and $A$ an ample $\mathbb{Q}$-divisor on $X$. Moreover, assume $(X,B)$ is klt. If $K_\mcF +A+B$ is nef, then it is semiample. 
    
\end{conjecture}

This can be thought of as a basepoint free theorem for foliated dlt pairs. Indeed, the above statement can be phrased in the following equivalent form: 

Let $(\mcF, B)$ be a foliated dlt pair on $X$ where $\mcF$ is algebraically integrable and $L$ a nef divisor on $X$ such that $L-(K_\mcF+B)$ is ample. Then $L$ is semiample. 


This conjecture was shown to be true for all corank one foliations on threefolds  by Cascini and Spicer \cite[Theorem 9.4]{CS}. In this article, we prove this conjecture.

\begin{theorem}[=Theorem \ref{CSconj}]
     Let $\mcF$ be an algebraically integrable foliation on a $\mathbb{Q}$-factorial normal projective variety $X$ which is induced by a dominant rational map $f:X \dashrightarrow Y$. Suppose $(\mcF, B)$ is F-dlt, $(X,B)$ is klt and $A$ an ample $\mathbb{Q}$-divisor on $X$. If $K_{\mcF}+A+B$ is nef, then it is semiample.

\end{theorem}
A direct consequence of this result is the contraction theorem for F-dlt pairs (sse Corollary \ref{cont}). We now come to some extensions of this theorem. Essentially the same arguments used in proving Theorem \ref{CSconj} show that in the given setting, if $K_\mcF+A+B$ is not assumed to be nef, then $P_\sigma(\mu^*(K_\mcF+A+B))$ is semiample for some birational morphism $\mu:X' \to X$. At present, we do not know how to prove Theorem \ref{CSconj} when $(\mcF,B)$ is only lc. However, in case $\mcF$ is induced by an equidimensional morphism $f: X \to Y$ such that if $B_f$ and $G$ denote the reduced branch and ramification divisors of $f$ respectively, $(X,B+G-\epsilon f^*B_f)$ is klt for some $\epsilon >0$ (for example, if $f$ is toroidal), then $K_\mcF+A+B$ is semiample. This result (see Lemma \ref{ab}) is a key ingredient in the proof of Theorem \ref{CSconj}. Proving it requires working in the category of foliated generalized lc pairs. We briefly mention the motivation behind using this larger category. One is the failure of Bertini-type results: if $(\mcF,B)$ is lc, then $(\mcF,B+A)$ need not be lc even if $A$ is general in its linear system. But since $(\mcF,B,\A)$ is generalized lc, this provides an easy way of working with $K_\mcF+A+B$. The other reason is that the proof of Lemma \ref{ab} involves a canonical bundle formula where the output is a foliated glc pair. Therefore, along the way, we extend several results for foliated lc pairs (proved in \cite{ACSS}, \cite{CS}, \cite{Liu}) namely adjunction, cone theorem and canonical bundle formula to the setting of foliated glc pairs. \\

We close this section with an interesting application of our results to the moduli part of a nice lc-trivial fibration:

\begin{theorem}(=Corollary \ref{mod})
    Let $(X,B)$ be a $\mbQ$-factorial klt pair, $A$ an ample $\mbQ$-divisor on $X$. Suppose $f: (X, \Sigma_X) \to (Y, \Sigma_Y)$ is an equidimensional toroidal morphism to a smooth projective variety such that $K_X+A+B \sim _{\mbQ, f}0$ and components of $B$ are components of the $f$-horizontal part of the toroidal boundary $\Sigma_X$. . Then the moduli part $M_X^{A+B}$ (see Definition \ref{moddef}) is semiample.
\end{theorem}
\textbf{Postscript:} This article replaces the preliminary version of our work which appeared on the arxiv in July, 2023. The main results are mostly unchanged, but due to issues with Lemma 2.4 and Theorem 3.5 in the previous version, some of the proofs have been modified. In addition, many of the proofs (particularly that of the Cone theorem) are expanded and rewritten which hopefully makes them easier to follow. 
\footnote{After the first version of this article appeared, G. Chen, J. Han, J. Liu and L. Xie have posted, a preprint \cite{CHLX} on the arxiv which generalizes many of our results to the setting of $\mbR$-divisors. This work is independent of theirs.}
 


{\bf Acknowledgments:} The first author thanks Najmuddin Fakhruddin and Roberto Svaldi for answering his questions and INdAM for financial support. We thank Jihao Liu for pointing out some errors in the first version of this article and his very detailed comments. In addition, we thank Paolo Cascini, Calum Spicer, Luca Tasin, Guodu Chen, Jingjun Han and Lingyao Xie for helpful discussions.
\section{Preliminaries}

\begin{definition} [Basics on Foliations; see \cite{Spi}, \cite{Dru}]

Let $X$ be a normal projective variety. A \emph{foliation} $\mcF$ on $X$ is a coherent subsheaf of the tangent sheaf $T_X$ which is closed under Lie brackets and such that $T_X/\mcF$ is torsion free. The \emph{singular locus} of $\mcF$ is the locus where $\mcF$ fails to be a sub-bundle of $T_X$. It is a large open subset of $X$ whose complement has codimension at least $2$. In particular, there exists a large open $U \subset X$ where $X$ and $\mcF$ are both smooth. The \emph{canonical divisor} of $\mcF$, denoted $K_\mcF$ is then the Zariski closure of det$(\mcF|_U)^*$. A subvariety $W \subset X$ is called \emph{$\mcF$-invariant} if for any local section $\partial$ of $\mcF$ over some $U \subset X$ open, $\partial(I_{W\cap U}) \subset I_{W \cap U}$. If $P \subset X $ is a prime divisor, then we define $\epsilon(P)=0$ if $P$ is $\mcF$-invariant and $\epsilon(P)=1$ otherwise.\\

Let $f:X \dashrightarrow Y$ be a dominant rational map between normal varieties and $\mcF$ a foliation on $Y$. Then as in \cite[Section 3.2]{Dru}, we can define the pullback foliation $f^{-1}\mcF$. The pullback of the zero foliation on $Y$ is known as the foliation induced by $f$. Such foliations are called \emph{algebraically integrable}. If $\mcF$ is the foliation induced by $f$, then one can check that a prime divisor $D \subset X$ is $\mcF$-invariant if and only if it is $f$-vertical, i.e. $f(D)$ does not dominate $Y$. If $f$ is a morphism, note that the foliation induced by $f$ is nothing but the relative tangent sheaf $T_{X/Y}$. If $f: X \dashrightarrow Y$ is birational and $\mcF$ a foliation on $X$, then we have an induced foliation on $Y$ defined by $g^{-1}\mcF$ where $g:= f^{-1}$.\\

If $f:X \to Y$ is an equidimensional morphism of normal varieties and $\mcF$ the foliation induced by $f$, then $K_\mcF \sim K_{X/Y}-R_f$, where $R_f := \sum_D(f^*D-f^{-1}D)$ is the ramification divisor of $f$.\\

Let $\mcF$ be the foliation on $X$ induced by a dominant rational map $f: X \dashrightarrow Y$. We say that a subvariety $V \subset X$ is \emph{tangent} to $\mcF$ if there exist 
\begin{enumerate}
    \item a birational morphism $\mu: X' \to X$ and an equidimensional contraction $X' \to Y'$ which induces $\mu^{-1}\mcF$; and
    \item a subvariety $V'$ contained in a fiber of $X' \to Y'$ such that $V= \mu(V')$.
\end{enumerate}
\end{definition}

We now explain how to extend the classical definitions of singularities of generalized pairs \cite[Section 4]{BZ} to the foliated setting.

\begin{definition}[Generalized foliated pairs and their singularities]\cite{Liu}]\label{def:g-pair} A \emph{generalized foliated pair} $(\mcF, B,\M)$ on a normal projective variety $X$ consists of a foliation $\mcF$ on $X$ a $\mathbb{Q}$-divisor $B \geq 0$ and a b-nef $\mathbb{Q}$-divisor $\M$ on X such that $K_\mcF+B+\M_X$ is $\mathbb{Q}$-Cartier. By abuse of notation we will sometimes use $\M$ to mean $\M_X$ from now on. Let $\pi:Y \to X$ be a higher model of $X$ to which $\M$ descends and $\mcF_Y$ the pulled back foliation on $Y$. Define $B_Y$ by the equation \begin{center}
    $K_{\mcF_Y}+B_Y+\M_Y = \pi^*(K_\mcF+B+\M)$,
\end{center} 
where $\mcF_Y:=\pi^{-1}\mcF$. If for any prime divisor $E$ on any such $Y$, mult$_EB_Y \leq \epsilon(E)$, then we say that $(\mcF,B,\M)$ is \emph{generalized log canonical}. We will use the abreviation \emph{glc} for generalized log canonical.\\

If $(\mcF, B,\M)$ (resp. $(X,B,\M)$) is a generalized foliated (resp. generalized) pair, we will denote by nglc$(\mcF, B,\M)$ (resp. nglc$(X,B,\M)$) the union of all non glc centers. Similarly, nlc will denote the union of all non lc centers. We will abreviate generalized log canonical center (foliated or otherwise) by \textit{glcc}.\\

\begin{remark}
It may seem more logically accurate to call $(\mcF,B,\M)$ a foliated triple than a generalized foliated pair. However, we chose the latter terminology to remain consistent with the classical literature on generalized pairs. 
\end{remark}

\end{definition}

The following is a particularly nice class of morphisms for which one can run a foliated MMP:

\begin{definition}[Property($*$) of \cite{ACSS}] \label{prop*} A log canonical pair $(X,B)/Z$ defined by a contraction $f: X \to Z$ is said to have \emph{Property ($*$)} if the following two properties are satisfied:
\begin{enumerate}
    \item there exists a reduced divisor $\Sigma_Z$ on $Z$ such that $(Z, \Sigma_Z)$ is log smooth and the $f$-vertical part of $B$ coincides with $f^{-1}\Sigma_Z$; and
    \item for any closed point $z \in Z$ and for any reduced divisor $\Sigma \geq \Sigma_Z$ such that $(Z,\Sigma)$ is log smooth around $z$, we have $(X, B+f^*(\Sigma - \Sigma_Z))$ is log canonical around $f^{-1}(z)$.
\end{enumerate}

If $(\mcF, \Delta)$ is a foliated pair, where $\mcF$ is the foliation given by an equidimensional contraction $f:X \to Z$, then $(\mcF, \Delta)$ is said to have \emph{Property ($*$)} if $(X, \Delta + G)/Z$ does, where $G:=f^{-1}B_f$ is the inverse image of the branch divisor of $f$. .
    
\end{definition}

A variety $X$ is said to be \emph{of klt type} if $(X, \Delta)$ is klt for some $\Delta \geq 0$.\\

The following lemmas will be used throughout this article.
\begin{lemma}\label{lem:equidimensional-adjunction}
    Let $f:X\to Y$ be an equidimensional proper surjective morphism from normal variety $X$ to a smooth variety $Y$ and $R(f)=\sum_P\left(f^*P-(f^*P)_{\red}\right)$ the ramification divisor, where $P$ runs through all prime divisors of $Y$. Suppose that $\mcF$ is a foliation induced by the morphism $f$. Then the following holds:
    \begin{equation}\label{eqn:ramification}
         K_{X/Y}\sim K_{\mcF}+R(f).
    \end{equation}
Moreover, if $\Delta$ is a $\mbQ$-divisor on $X$, then there is a $f$-vertical reduced divisor $\Gamma$ such that
\begin{equation}\label{eqn:relative-pair-to-foliation}
    K_X+\Delta+\Gamma\sim_{f, \mbQ} K_{\mcF}+\Delta.
\end{equation}
\end{lemma}

\begin{proof}
    The relation \eqref{eqn:ramification} follows from \cite[Example 3.2]{Dru}. Let $B:=(f(R(f)))_{\red}$ and $\Gamma:=(f^*B)_{\red}$, then $B$ is a reduced divisor on $Y$ such that $R(f)=f^*B-(f^*B)_{\red}=f^*B-\Gamma$. Thus from \eqref{eqn:ramification} we have $ K_X+\Gamma\sim_f K_{\mcF}$. Adding $\Delta$ both sides proves our claim.   
\end{proof}

\begin{lemma}\label{lem:toroidal-singularity}
    Let $(X, \Sigma)$ be a toroidal pair (see \cite[Definition 2.1]{ACSS}) and $B$ a $\mbQ$-divisor such that $\Supp(B)\subset \Sigma$ and $K_X+B$ is $\mbQ$-Cartier. Then $(X, B)$ has log canonical (resp. klt) singularities if and only if the coefficients of $B$ are contained in the interval $[0, 1]$ (resp. $(0, 1)$).\\ 
    In parituclar, if $\mbf M$ is a $b$-nef $\mbQ$-divsor such that $\mbf M$ descends on $X$, then $(X, B, \mbf M)$ has glc (resp. gklt) singularities if and only if the coefficients of $B$ are contained in the interval $[0, 1]$ (respl. $(0, 1)$).
\end{lemma}

\begin{proof}
Since the question is (analytic) local on $X$ and $(X, \Sigma)$ is a toroidal pair, replacing $X$ by a small analytic open subset we may assume that $X$ is a toric varity and $\Sigma$ is a torus invarient divisor on $X$. Then the singularity criteria of $(X, B)$ follows from \cite[Proposition 11.4.24]{CLS11}. In the generalized pair case, since $\mbf M$ descends to $X$, then generalized singulairties of $(X, B, \mbf M)$ are equivalent to the corresponding usual singularities of $(X, B)$ and the claim follows.  
\end{proof}

\section{Properties of generalized foliated pairs}

\subsection{Adjunction} In this subsection we will prove an adjunction theorem for generalized foliated pair. This is a generalization of \cite[Proposition 3.2]{ACSS} and our arguments are similar to theirs. 

\begin{theorem} \label{thm:adjunction}
    Let $\mathcal{F}$ be a foliation on a normal projective variety $X$ induced by a contraction $f : X \rightarrow Z$. Suppose that $(\mathcal F, \Delta , \mathbf{M})$ is a foliated generalized log canonical pair on $X$ and $X$ has a small birational model which is $\mathbb{Q}$-factorial. Let $T$ be a prime Weil divisor on $X$ with normalization $\nu: S \to T$ such that $\mult_T(\Delta)=\epsilon(T)$. Then
    \begin{enumerate}
        \item there exists a generalized foliated pair $(\mathcal{F}_S, \Delta_S,\mathbf{N})$ on $S$ such that 
        \begin{center}
            $\nu^*(K_{\mathcal{F}}+\Delta+\mathbf{M})= K_{\mathcal{F}_S}+\Delta_S+\mathbf{N}$;
        \end{center}
        \item if $(\mathcal{F}, \Delta, \mathbf{M})$ is generalized log canonical, then so is $(\mathcal{F}_S, \Delta_S,\mathbf{N})$.
    \end{enumerate}

    \begin{proof}
    First replacing $X$ by a small $\mbQ$-factorization we may assume that $X$ is $\mbQ$-factorial. We  
    define $\Delta_S$ as in \cite[Defintion 4.7]{BZ}. Now we will show that $\Delta_S \geq 0$ is effective. Since $X$ is $\mbQ$-factorial, $\mathbf{M}_X$ is $\mathbb{Q}$-Cartier. Let $\pi:X' \to X$ be a log resolution such that $(X', \Ex(\pi)+ \pi_*^{-1}\Delta)$ is log smooth, $\mathbf{M}$ descends to $X'$ and $S':=\pi_*^{-1}T$ is smooth.
       Let $\tau: S'\to S$ be the induced morphism and $\mathbf{N}_S:=\tau_*(\mathbf{M}_{X'}|_{S'})$. If $\Tilde{\Delta}_S$ is defined by usual adjunction  $\nu^*(K_\mcF+\Delta)= K_{\mcF_S}+\Tilde{\Delta}_S$ (see \cite[Proposition 3.2]{ACSS}), then we claim that $\Delta_S \geq \Tilde{\Delta}_S$. Indeed, we have
       \begin{equation}\label{eqn:different}
         \Delta_S-\Tilde{\Delta}_S+\mathbf{N}_S=\nu^*\mathbf{M}_X. 
       \end{equation}
  By the negativity lemma, $\pi^*\mathbf{M}_X=\mathbf{M}_{X'}+E$ for some effective $\pi$-exceptional $\mbQ$-divisor $E \geq 0$. Restricting to $S'$ and applying $\tau_*$ we obtain that $\nu^*\mathbf{M}_X =\mathbf{N}_S+ \tau_*(E|_{S'})$. In particular,  $\nu^*\mathbf{M}_X \geq \mathbf{N}_S$, and hence by \eqref{eqn:different}, $\Delta_S \geq \Tilde{\Delta}_S$.\\
  But $\Tilde{\Delta}_S \geq 0$ by \cite[Proposition 3.2]{ACSS}, and thus $\Delta_S \geq 0$. \\

       We are left to show that $(\mcF_S, \Delta_S,\mathbf{N})$ is glc. For this, we argue along the lines of \cite[Proposition 3.2]{ACSS}.\\

       \underline{Case 1:} $\epsilon(T)=0$, i.e. $T$ is $\mcF$-invariant. First we recall from \cite{ACSS} the definition of $\mathcal{F}_S$. Define $U:=X\setminus(X_{\sing}\cup T_{\sing}\cup \mcF_{\sing})$. Then since $T$ is $\mathcal{F}$-invariant, the inclusion $\mcF|_{T\cap U}\to T_X|_{T\cap U} $ factors through $T_{T\cap U}$. This defines a foliation on $T\cap U$ and hence a foliation $\mcF_S$ on $S$.\\
       
       Let $\pi:\overline{X} \to X$ be a log resolution of $(X, \Delta)$ on which $\mathbf{M}$ descends. Then by \cite[Theorem 2]{Ka}, there exist birational morphisms from normal projective varieties $X' \to \overline{X}$ and $Z'\to Z$ ($Z'$ smooth) with an induced equidimensional toroidal contraction $f': (X', \Sigma_{X'}) \to (Z', \Sigma_{Z'})$ such that if $\beta: X' \to X$ is the composite morphism, then we have $\mbox{Supp}(\beta_*^{-1}(B+S))+ \Ex \beta \subset \Sigma_{X'}$ and $f'$ induces the foliation $\mcF':=\beta^{-1}\mcF$. Since all toroidal pairs here are without self-intersection (see \cite[Section 1.3]{AK}),  by \cite[Proposition-Definition 2, page 57]{KKFMS73}, $S':= \beta_*^{-1}S$ is normal. Let $T'$ be the normalization of $f'(S')$. Then the restricted foliation $\mcF_{S'}$ is induced by $f'|_{S'}: S' \to T'$ (this follows from the definition of $\mcF_{S'}$). Note that $f'|_{S'}:S'\to T'$ is also toroidal between the induced toroidal pairs, see \cite[Prop. 11.4.24]{CLS11} \\
       
       We can write $\beta^*(K_\mcF +\Delta+\M) = K_{\mcF'}+\Delta'+F_0+F_1+\M$, where $\Delta':=\beta_*^{-1}\Delta$, $F_0 \leq 0$ is $\mcF'$-invariant and $F_1$ is non-$\mcF'$-invariant and the coefficients of $F_1$ are at most $1$. Then we have $K_{\mcF'}|_{S'}=K_{\mcF'_{S'}}+ \Theta$, where $\Theta \geq 0$ and $\Theta$ is scheme theoretically equal to the codimension $2$ part of Sing $\mcF'$ contained in $S'$; see the proof of \cite[Proposition 3.6]{Dru}. Now, Sing $\mcF'$, being the singular locus of the foliation induced by a toroidal morphism, is reduced and torus invariant. This implies that so is $\Theta$. Moreover, $f'|_{S'}$(Supp $\Theta) =T'$ which follows from the fact that a foliation can not be generically singular along an entire leaf. In particular, $\Theta$ is not $\mcF'_{S'}$ invariant. \\
       
       We have $
        \beta^*((K_\mcF+\Delta+\M)|_S)= 
        K_{\mcF'_{S'}}+ \Theta+(\Delta'+F_1)|_{S'}+F_0|_{S'}+\N_{S'}.
        $
       Note that $\Delta'+F_1$ contains no codimension $2$ components of Sing $\mcF'$ (Indeed, if $W$ is a codim $2$ component of Sing $\mcF'$ contained in $\Delta'+F_1$, then there exist divisors $L_1$ and $L_2$ which are leaves of $\mcF'$ and intersect along $W$. Now, $(\mcF', \Delta'+F_1)$ is sub-lc iff $(X',L_1+L_2+\Delta'+F_1)$ is sub-lc. This follows from arguments in the proof of \cite[Lemma 3.3]{CS}. But, $(X', L_1+L_2+\Delta'+F_1)$ is not sub-lc, while $(\mcF', \Delta' +F_1)$ is sub-lc by \cite[Lemma 3.1]{ACSS}.) and no component of $(\Delta'+F_1)|_{S'}$ can be $\mcF'_{S'}$-invariant. Thus we can conclude that $(\mcF', \Theta+(\Delta'+F_1)|_{S'}+F_0|_{S'}+ \M)$ is sub-glc by \cite[Lemma 3.1]{ACSS}.\\
       
       \underline{Case 2:} $\epsilon(T)=1$. Let $\pi:\overline{X} \to X$ be a higher model of $X$ on which $\mathbf{M}$ descends and such that $(\overline{X}, \pi_*^{-1}\Delta+ \Ex \pi)$ is log smooth. By \cite[Theorem 2.1]{AK}, there exist birational morphisms $X' \to \overline{X}$ and $Z' \to Z$ such that the induced morphism $(X', \Sigma_{X'}) \to (Z', \Sigma_{Z'})$ is a toroidal morphism between log smooth toroidal pairs such that letting $\beta:X' \to X$ denote the induced morphism, we have $\beta_*^{-1}\Delta+ \Ex \beta \subset \Sigma_{X'}$. Then the induced morphism $(S', \Sigma_{X'}|_{S'}) \to (T', \Sigma_{Z'}|_{T'}) $ is toroidal. Moreover, $f'$ has good horizontal divisors (see the proof of \cite[Theorem 9.5, page 61]{Kar}). More precisely, there exist local models $X_\sigma$ of $X'$ and $X_{\sigma'}$ for $Z'$ with $X_\sigma \cong X_{\sigma'}\times \mathbb{A}^m$ such that the horizontal divisors of $f'$ are pullbacks of coordinate hyperplanes of $\mathbb{A}^m$. In particular, $S'$ is transverse to the fibers of $f'$. By \cite[Corollary 3.3]{Spi}, we have $(K_{\mcF'}+S')|_{S'}= K_{\mcF'_{S'}}+\Theta_{S'}$ where $\Theta_{S'} \geq 0$ is tangent to $\mcF'$. But then $\Theta_{S'}=0$. As above, we write $\beta^*((K_\mcF+\Delta+\M)|_S)= K_{\mcF_{S'}}+ (\Delta'-S'+F_1)|_{S'}+F_0|_{S'}+\N_{S'}$. Since $f':X' \to Z'$ has good horizontal divisors, no component of $\Delta'-S'+F_1$ is tangent to $\mcF'$ anywhere. Thus $(\Delta'-S'+F_1)|_{S'}$ is $f'|_{S'}$-horizontal, which shows that $(\mcF_{S'}, (\Delta'-S'+F_1)|_{S'}+F_0|_{S'}+\N)$ is sub-glc again by \cite[Lemma 3.1]{ACSS}.

    \end{proof}
\end{theorem}

\subsection{Cone theorem}

We need the following simple variation of the bend and break lemma of Spicer \cite[Corollary 2.28]{Spi}:

\begin{lemma}\label{bb} Let $(\mcF, \Delta , \mathbf{M})$ be a generalized foliated pair on a normal projective variety $X$ of dimension $n$. Let $N$ be any nef divisor on $X$. Suppose there exist nef $\mathbb{R}$-divisors $D_1, \cdots D_n$ such that $D_1 \cdots D_n =0$ and $-(K_\mcF + \Delta +\mathbf{M}_X) \cdot D_2 \cdots D_n >0$. Then through a general point of $X$, there passes a rational curve $\Sigma$ tangent to $\mcF$ such that $D_1 \cdot \Sigma =0$ and
$N \cdot \Sigma \leq \frac{N \cdot D_2 \cdots D_n}{-K_\mcF \cdot D_2 \cdots D_n}$.

\begin{proof}
    The arguments of \cite[Corollary 2.28]{Spi} still work if we replace $\Delta$ there with $\Delta+\mathbf{M}_X$. 
\end{proof}

\end{lemma}

\begin{theorem}\label{cone}  Let $(\mathcal{F}, \Delta ,\mathbf{M})$ be a generalized foliated log canonical pair on a normal projective variety $X$, where $\mathcal{F}$ is induced by a dominant map $f: X \dashrightarrow Z$ of normal varieties. Then there are at most countably many rational curves $\{\xi_i\}_{i\in I}$ such that $\xi_i$ is tangent to $\mcF$, $0<-(K_{\mcF}+\Delta+\bfM_X)\cdot\xi_i\<2n$ for all $i\in I$ and 
\begin{equation*}
\overline{NE}(X)= \overline{NE}(X)_{K_{\mathcal{F}}+\Delta+ \mathbf{M}_X \geq 0}+ \sum \mathbb{R}_+[\xi _i]. 
\end{equation*}
\end{theorem}

\begin{proof}
Our proof uses some ideas of the proof of \cite[Theorem 3.9]{ACSS}.
We will proceed by induction on the dimension of $X$. Assume that the above theorem holds for varieties of dimension $\<n-1$. Let $R \subset \overline{NE}(X)$ be a $K_\mcF+\Delta+\bfM_X$-negative exposed extremal ray supported by the nef (but not ample) Cartier divisor $H_R$, i.e. $H_R^\bot\cap\NE(X)=R$. We make the following claim:

\begin{claim}\label{clm:existence-of-curve}
     $R$ is generated by a rational curve $\xi\subset X$ tangent to $\mcF$ such that $0<-(K_X+\Delta+\mbf M_X)\cdot \xi\<2\dim X$.
\end{claim}

\begin{proof}[Proof of Claim \ref{clm:existence-of-curve}]
 From a standard argument it follows that there is an ample $\mbQ$-divisor $A$ such that $H_R\num K_{\mathcal{F}}+\Delta+\mathbf{M}_X+A$. We divide our proof into two cases depending on whether $H_R$ is big or not.\\ 

\noindent
\textbf{Case I:} Suppose that $H_R$ is not big. In this case we argue as in the proof of \cite[Theorem 3.9]{ACSS}. Let $\nu:= \nu(H_R)$ be the numerical dimension of $H_R$, $D_i:=H_R$ for $1 \leq i \leq \nu+1$ and $D_i:=A$ for $\nu+1 < i \leq n$. Then $D_1 \cdots D_n = H_R^{\nu+1} \cdot A^{n- \nu-1} =0$ and thus $-(K_\mcF+\Delta+\mathbf{M}_X) \cdot D_2 \cdots D_n = (A-H_R) \cdot H_R^\nu \cdot A^{n-\nu-1} >0$. Fix $k \gg 0$ so that $N:=kH_R-(K_\mcF+\Delta+\mathbf{M}_X)$ is ample. Then by Lemma \ref{bb}, through a general point of $X$ there is a rational curve $\xi$ tangent to $\mcF$ such that $H_R\cdot \xi =0$ (thus $\xi \in R$) and 
\begin{displaymath}
\begin{split}
    -(K_\mcF+\Delta+\mathbf{M}_X)\cdot \xi = N \cdot \xi & \leq 2n \left(\frac{N \cdot H_R^\nu \cdot A^{n- \nu-1}}{-K_\mcF \cdot H_R^\nu \cdot A^{n-\nu-1}}\right)\\
    &=2n \left(\frac{-(K_\mcF+\Delta+\mathbf{M}_X)\cdot H_R^\nu \cdot A^{n-\nu-1}}{-K_\mcF \cdot H_R^\nu \cdot A^{n-\nu-1}}\right)\\
    &\leq 2n ,
    \end{split}
\end{displaymath}
where the last inequality follows from the fact that $(\Delta+ \mathbf{M}_X)\cdot H_R^\nu\cdot A^{n-\nu-1}\>0$. Indeed, to see that $\bfM_X\cdot H^\nu_R\cdot A^{n-\nu-1}\>0$, let $g:X'\to X$ be a birational morphism so that $\bfM$ descends on $X'$ and $\bfM_{X'}$ is nef. Then we have $\bfM_X\cdot H^\nu_R\cdot A^{n-\nu-1}=\bfM_{X'}\cdot (g^*H_R)^\nu\cdot (g^*A)^{n-\nu-1}\>0$.\\

\noindent
\textbf{Case II:} Suppose that $H_R$ is big. Then $H_R\equiv H'+E$ for some ample $\mbQ$-divisor $H'$ and an effective $\mbQ$-divisor $E$. Since $H_R\cdot R=0$, it follows that $E\cdot R<0$, and hence $S\cdot R<0$ for some irreducible component $S$ of $E$. For a subvariety $W\subset X$ with normalization $W^ \nu$ we will write $R \subset W$ to mean that $R \subset \im(\NE(W^\nu) \to \NE(X))$. Since $S\cdot R<0$, it's not hard to see that $R\subset S$.\\
Now let $\Lambda$ be the set of pairs $(W, \lambda)$ satisfying the following properties:
\begin{enumerate}
    \item $\lambda\>0$,
    \item $W\subset X$ is a generalized log canonical center of $(\mcF, \Delta+\lambda S, \mbf M)$, and
    \item $R\subset W$.
\end{enumerate}
Note that $\Lambda$ is not an empty collection as $(S, \lambda')\in\Lambda$ for some $\lambda'\>0$. Let $(W_0, \lambda_0)$ be a minimal element of $\Lambda$ with respect to inclusion in the first entry. \\
Let $\mu:\tilde{X} \to X$ be a resolution of $X$ such that $\mathbf{M}$ descends to $\tilde{X}$ and $\mu^{-1}W_0$ is a SNC divisor on $\tilde{X}$  containing a generalized log canonical place of $(\mcF, \Delta+\lambda_0S, \mbf M)$ over $W_0$. By \cite[Theorem 2.2]{ACSS} or \cite[Theorem 2]{Ka}, there exist birational morphisms $X' \rightarrow \Tilde{X}$, $Z' \rightarrow Z$ and reduced divisors $\Sigma_{X'}$ and $\Sigma_{Z'}$ on $X'$ and $Z'$ respectively and equidimensional toroidal morphism $f':(X', \Sigma_{X'}) \to(Z', \Sigma_{Z'})$ such that $(Z', \Sigma_{Z'})$ is log smooth, $\beta_*^{-1}(\Delta + S) + \Ex(\beta) \subset \Sigma_{X'}$ and $\mathcal{F'}:= \beta ^{-1}\mathcal{F}$ is induced by $f'$, where $\beta:X'\to X$ is the composite morphism. Let $E_i$ be the $\beta$-exceptional divisors. If $\epsilon(S)=1$, we define $\mu\>0$ as follows
\begin{equation}
    \begin{split}
        \mu &:= \sup \left\{t\>0\;:\; \left(\mathcal{F}', \beta_*^{-1}(\Delta+tS) + \sum \epsilon(E_i)E_i, \mathbf{M}\right) \mbox{ is glc}\right\}\\
            &=\sup\left\{t\>0\;:\; \left(\mathcal{F}', \beta_*^{-1}(\Delta+tS) + \sum \epsilon(E_i)E_i\right) \mbox{ is lc}\right\} \quad (\mbox{as } \mbf M \mbox{ descends to } X').
    \end{split}
\end{equation}
Then by \cite[Lemma 3.1]{ACSS}, $\mu=1-\mult_S(\Delta)$ and $(\mathcal{F}', \beta_*^{-1}(\Delta+ \mu S) + \sum \epsilon(E_i)E_i)$ is lc, in particular,  $(\mathcal{F}', \beta_*^{-1}(\Delta+ \mu S) + \sum \epsilon(E_i)E_i, \mathbf{M})$ is glc, as $\mbf M$ descends to $X'$. Now by Lemma \ref{lem:equidimensional-adjunction}, there is a $f$-vertical reduced divisor $\Gamma\subset \Sigma_{X'}$ on $X'$ such that 
\begin{equation}\label{eqn:canonical-comparison}
    K_{X'}+\Gamma \sim_{\mbQ, Z'} K_{\mathcal{F'}}.
\end{equation}
By Lemma \ref{lem:toroidal-singularity}, $(X', \beta_*^{-1}(\Delta+ \mu S)+ \sum \epsilon(E_i)E_i+ \Gamma, \mathbf{M})$ has glc singularities.
If $\epsilon(S)=0$, then clearly the same conclusion holds for $\mu=0$.\\

Next we will show the existence of a special model for the foliated pair $(\mcF, \Delta+\mu S, \mbf M)$ (called a Property ($*$) model in \cite{ACSS}) which is somewhat similar to the dlt model of usual pairs.

\begin{claim}\label{clm:dlt-model}
  There is a $\mbQ$-factorial birational model $\beta'':X''\to X$, effective $\mbQ$-divisors $\Theta''_\mu$ and $F_\mu$ such that $(\mcF'', \Theta''_\mu, \M)$ has glc singularities, $\beta''(F_\mu)\subset {\rm nglc}(\mcF, \Delta+\mu S, \M)$, and
  \begin{equation}
      K_{\mcF''}+\Theta''_{\mu}+F_{\mu}+\M_{X''}=\beta''^*(K_{\mcF}+\Delta+\mu S+\M_X),
  \end{equation}
  where $\mcF''$ is the pulled back foliation on $X''$.
\end{claim}

\begin{proof}[Proof of Claim \ref{clm:dlt-model}]
Let $H_{Z'}$ be an ample Cartier divisor on $Z'$, $\Theta':=\beta_*^{-1}(\Delta+ \mu S) + \sum \epsilon(E_i)E_i$, and $n=\dim X'=\dim X$. By Bertini's theorem, there is a semi-ample $\mbQ$-Cartier divisor $H'\sim_{\mbQ} (2n+1)f'^*H_{Z'}$ such that $(X', \Theta'+\Gamma+H',\M)$ is glc. The idea is to run a $K_{X'}+\Theta'+\Gamma+H'+\bfM_{X'}$-MMP over $X$ with scaling of an ample divisor $A$ (using \cite[Theorem 1.3]{HL21}) to obtain the required dlt model. 
First, we claim that the curves contracted at every step of this MMP are vertical over $Z'$; in particular, every step of this MMP is also a step of $K_{\mathcal{F'}}+\Theta'+H'+\mathbf{M}_{X'}$-MMP over $X$.\\

To see this, we proceed by induction on the indices $i$ and assume that it holds true for all steps $j\leq i-1$. Let $\varphi_i:X'_i\to Y'_i$ be the $i$-th contraction (divisorial or flipping type) of the $K_{X'}+\Theta'+\Gamma+H'+\bfM_{X'}$-MMP over $X$ corresponding to the extremal ray $R_i$. Let $f'_i:X'_i\to Z'$ be the induced morphism and $H'_i:={f'}_i^*H_{Z'}$. Then $(K_{X'_i}+\Theta'_i+\Gamma_i+\bfM_{X'_i})\cdot R_i<0$ (as $H'_i$ is nef) and by \cite[Theorem 1.3]{HL21} there is a rational curve $C_i\subset X'_i$ contained in the fiber of $X'_i\to X$ such that $R_i=\mbR_+[C_i]$ and $(K_{X'_i}+\Theta'_i+\Gamma_i+\bfM_{X'_i})\cdot C_i\>-2n$. Thus $f'_{i,*}(C_i)=0$, otherwise $(K_{X'_i}+\Theta'_i+\Gamma_i+H'_i+\bfM_{X'_i})\cdot C_i=(K_{X'_i}+\Theta'_i+\Gamma_i+(2n+1){f'}_i^*H_{Z'}+\bfM_{X'_i})\cdot C_i>0$, a contradiction. Then from \eqref{eqn:canonical-comparison}  we have
$
(K_{\mcF'_i}+\Theta'_i+H_i'+\bfM_{X'_i})\cdot C_i=(K_{X'_i}+\Theta'_i+\Gamma_i+H'_i+\bfM_{X'_i})\cdot C_i<0,
$
which shows that $\vphi:X'_i\to Y'_i$ is a $(K_{\mcF'_i}+\Theta'_i+H_i'+\bfM_{X'_i})$-negative contraction. This proves our claim.\\

Now we return to the proof of Claim \ref{clm:dlt-model}. We can write $K_{\mcF'}+\Theta'+\M_{X'} = \beta^*(K_\mcF+\Delta+ \mu S+\M_X)+E_1-E_2$, where $E_1$, $E_2$ are effective $\beta$-exceptional divisors without common components. Clearly, $\beta(E_2) \subset{\rm nglc}(\mcF, \Delta+\mu S, \M)$. 

Let $\lambda_1:=$ inf $\{t|K_{\mcF'}+\Theta'+H'+\M_{X'}+tA$ is nef over $X\}$. We first show that there exists an extremal ray $R_1 \subset \overline{NE}(X'/X)$ such that $(K_{\mcF'}+\Theta'+H'+\M_{X'})\cdot R_1<0$ and $(K_{\mcF'}+\Theta'+H'+ \M_{X'}+\lambda_1 A)\cdot R_1 =0$. For this, we may assume $\lambda_1 >0$. Choose $\mu_1 < \lambda_1$. Then $K_{\mcF'}+\Theta'+H'+\mu_1 A+\M_{X'}$ is not nef over $X$, so has a negative extremal ray, say $R_1$. Then $R_1$ is also $K_{\mcF'}+\Theta'+\M_{X'}$-negative. Let $S':= \beta_*^{-1}S$.
\begin{claim}\label{tangency}
    For any $K_{\mcF'}+\Theta'+\M_{X'}$-negative extremal ray $R_1$ over $X$, $R_1$ is spanned by a rational curve tangent to $\mcF'$.
\end{claim}
\begin{proof}
    
Indeed, we have the following two possibilities for $R_1$:

\begin{enumerate}
    \item $S' \cdot R_1 <0$. Then $R_1 \subset \overline{NE}(S')$. By adjunction, $K_{\mcF'}+\Theta'+\M_{X'}$ restricts to a glc pair on $S'$. By cone theorem in dimension $n-1$, $R_1$ is spanned by a rational curve $C$ tangent to the restricted foliation $\mcF'_{S'}$ and hence to $\mcF'$.\vspace{2mm}

    \item $S' \cdot R_1 \geq 0$. In this case, letting $\Delta' := \beta_*^{-1}\Delta$, we have $(K_{\mcF'}+\Delta' + \sum \epsilon(E_i)E_i +\M_{X'}) \cdot R_1 <0$. Since $(\mcF, \Delta,\M)$ is glc, we can write $K_{\mcF'}+ \Delta' +\sum \epsilon (E_i)E_i +\M_{X'} =  \beta^*(K_\mcF+\Delta+\M_X)+\sum a_i E_i$, with $a_i \geq 0$ for all $i$ and not all $a_i$ being $0$. Since $R_1 \subset \overline{NE}(X'/X)$, we get $E_i \cdot R_1<0$ for some $i$ and thus by  adjunction on $E_i$ and cone theorem in dimension $n-1$ as above, we get a rational curve $C$ spanning $R_1$ which is tangent to $\mcF'$.\\

This ends the proof of Claim \ref{tangency}.

\end{enumerate}
\end{proof}
In both of the above cases, we have the following length estimate for $C$ coming from cone theorem in dimension $n-1$: $(K_{\mcF'}+\Theta'+\M_{X'}) \cdot C \geq -2(n-1)$. 
\begin{claim}\label{scaling ray}
    There exists a  $K_{\mcF'}+\Theta'+H'+\M_{X'}+ \mu_1 A$-negative extremal ray over $X$ which is $K_{\mcF'}+\Theta'+H'+\M_{X'}+\lambda_1 A$-trivial. 
\end{claim}
\begin{proof}
The arguments are standard; see for example \cite[Lemma 3.10.8]{BCHM}, \cite[Lemma 9.2]{CS}. We give the details for the reader's convenience. The number of $K_{\mcF'}+\Theta'+H'+\M_{X'}+ \mu_1 A$-negative extremal rays over $X$ is finite. Indeed, they are $K_{\mcF'}+\Theta'+\M_{X'}$-negative, so by Claim \ref{tangency}, they are spanned by curves, hence discrete. Since $H'+\mu_1A$ is ample, the desired finiteness follows. \\

Let $R_1,\cdots, R_k$ be the $K_{\mcF'}+ \Theta'+ H'+ \M_{X'}+\mu_1A$-negative extremal rays over $X$, where $R_i= \mbR_{+}[C_i]$. Let 
\begin{center}
    $\lambda:= {\rm min} \frac{A \cdot C_i}{-(K_{\mcF'}+\Theta'+H'+\M_{X'}+\mu_1A) \cdot C_i}$.
\end{center}
Then $K_{\mcF'}+\Theta'+H'+\M_{X'}+\mu_1A+ \frac{1}{\lambda}A$ is nef over $X$ and for some $i$, we have 
\begin{center}
    $(K_{\mcF'}+\Theta'+H'+ \M_{X'}+\mu_1A +\frac{1}{\lambda}A) \cdot C_i=0$.
\end{center}
Thus we conclude that $\mu_1+ \frac{1}{\lambda}= \lambda_1$. This finishes the proof of Claim \ref{scaling ray}.
\end{proof}

Since $K_{\mcF'}+\Theta'+H'+\M_{X'} \sim _{\mbQ, Z'} K_{X'}+\Theta'+H'+ \Gamma+\M_{X'}$, where the latter is a glc pair, this extremal ray can be contracted. Let $X'_1:= X' \dashrightarrow X_2'$ be the corresponding step of MMP over $X$. This is the first step of the foliated MMP with ample scaling.\\

We show how to construct the $i$-th step of the MMP, given the $i-1$-th step, say $X'_{i-1} \dashrightarrow X'_i$. Let 
\begin{center}
    $ \alpha_i:=$inf $\{t|K_{X'_i}+ \Theta'_i+H'_i+ \Gamma_i +\M_{X'_i}+tA_i$ is nef over $X\}$, 

    $\lambda_i:= {\rm inf} \{t| K_{\mcF'_i} +\Theta'_i+H'_i+\M_{X'_i} +tA_i$ is nef over $X\}$,
\end{center}
 then we claim that $\lambda_i =\alpha_i$. Indeed, suppose $K_{X'_i}+ \Theta'_i+H'_i+\Gamma_i +\M_{X'_i}+t A_i$ is not nef over $X$. Let $R \subset \overline{NE}(X_i'/X)$ be a negative extremal ray with respect to it. Then $(K_{X'_i}+\Theta'_i+\Gamma_i+\M_{X'_i}+t A_i) \cdot R <0$. By the classical cone theorem, $R= \mbR_+[C]$, where $(K_{X'_i}+ \Theta'_i+ \Gamma_i + \M_{X'_i}+t A_i) \cdot C \geq -2n$. If $(H_i' \cdot C) >0$, then $(H'_i \cdot C) \geq 2n+1$, which would make $(K_{X'_i}+ \Theta'_i+\Gamma_i+H'_i+t A_i+\M_{X'_i}) \cdot C \geq 0$, a contradiction. Thus $H' \cdot C  =0$, which implies $f_i'(C)=$pt. Thus, $(K_{\mcF'_i}+\Theta'_i+ H'_i+ \M_{X'_i}+ t A_i)\cdot C = (K_{X'_i}+\Theta'_i+\Gamma_i +H'_i+\M_{X'_i} +t A_i) \cdot C$. Thus, we can conclude that $K_{\mcF_i'}+\Theta_i'+H'_i+\M_{X'_i}+tA_i$ is not nef over $X$. We conclude $\alpha_i \leq \lambda_i$.\\

If $\alpha_i < \lambda_i$, then $K_{\mcF'_i}+\Theta'_i+H'_i+ \M_{X'_i}+\alpha_i A_i$ is not nef over $X$. Then we have an extremal ray $R \subset \overline{NE}(X'_i/X)$ such that

\begin{equation} \label{blah}
    (K_{\mcF'_i}+ \Theta'_i+ H'_i+ \M_{X'_i}+ \alpha_i A_i) \cdot R <0.
\end{equation}
Since $(K_{\mcF'_i}+ \Theta'_i+H'_i+ \M_{X'_i} + \lambda_i A_i) \cdot R \geq 0$, we get $(A_i \cdot R) >0$. Thus equation \ref{blah} implies that $(K_{\mcF'_i}+\Theta'_i+H'_i +\M_{X'_i}) \cdot R<0$. In particular, $(K_{\mcF'_i}+\Theta'_i +\M_{X'_i}) \cdot R <0$. Then by the same arguments as above, $R$ is spanned by a rational curve tangent to $\mcF'_i$. This would make $(K_{\mcF'_i}+\Theta'_i+H'_i+ \M_{X'_i}+\alpha_i A_i) \cdot R = (K_{X'_i}+\Theta'_i+ H'_i+ \Gamma_i +\M_{X'_i}+ \alpha_i A_i) \cdot R \geq 0$ (since the latter divisor is nef over $X$), a contradiction. Thus $ \alpha_i = \lambda_i$.\\

We can thus run the desired foliated MMP with ample scaling, which is nothing but a $K_{X'}+ \Theta'+ H'+ \Gamma + \M_{X'}$-MMP over $X$ with scaling of $A$. Note that this MMP is $H'$-trivial as we observed in the beginning of proof of Claim \ref{clm:dlt-model}.\\

Let $ \lambda_{H'_i}:= {\rm inf}\{t|K_{\mcF_i'}+\Theta_i'+H_i'+\M_{X_i'}+tA_i $ is nef over $X\}$ and $\mu_0:={\rm lim} \lambda_{H'_i}$. Suppose $\mu_0 >0$. Then the $K_{\mcF'}+\Theta'+H'+\M_{X'}$-MMP over $X$ with scaling of $A$ is also a $K_{\mcF'}+\Theta'+H'+\M_{X'}+\frac{1}{2}\mu_0 A$-MMP over $X$. We now show that this MMP terminates. Indeed, otherwise, since this is also an MMP over $Z'$ (by Claim \ref{tangency}) and $K_{\mcF'}+\Theta'+H'+\M_{X'}+\frac{1}{2}\mu_0 A \sim _{Z'}K_{X'}+\Theta'+H'+\Gamma+\frac{1}{2}\mu_0 A+\M_{X'} \sim K_{X'}+D'$ with $(X',D')$ klt for some $D' \geq 0$ big, we obtain an infinite sequence of $K_{X'}+D'$-flips with ample scaling. This is ruled out by \cite{BCHM}. Thus, if $\mu_0 >0$, this MMP always terminates with $X' \dashrightarrow X''$ such that $K_{\mcF''}+\Theta''+ H'' +\M_{X''}$ is nef over $X$. This forces $K_{\mcF''}+\Theta''+\M_{X''}$ to be nef over $X$. Indeed, if not, there exists a negative extremal ray over $X$, say $R''$ with respect to it, which (as in the proof of Claim \ref{tangency}) is tangent to $\mcF''$. In particular, $H'' \cdot R'' =0$. Then $(K_{\mcF''}+\Theta''+ \M_{X''}) \cdot R'' = (K_{\mcF''}+\Theta''+ H''+\M_{X''}) \cdot R''$ which can't be negative since $K_{\mcF''}+\Theta''+H''+\M_{X''}$ is nef. \\

Let $\beta'': X'' \rightarrow X$ denote the induced morphism. Then $K_{\mcF''}+\Theta''+ \M_{X''} = \beta''^*(K_\mcF +\Delta +\mu S +\M_{X})+E_1''-E_2''$, where $E_1''$ and $E_2''$ are effective $\beta''$-exceptional divisors without common components. Since $K_{\mcF''}+\Theta''+\M_{X''}$ is nef over $X$, $E_1''=0$ by negativity lemma and this completes the proof of Claim \ref{clm:dlt-model} in the case $\mu_0 >0$. In any case, note that if the MMP with scaling terminates with $X' \dashrightarrow X''$, then $E_1''=0$.\\


We are left to consider the case $\mu_0 =0$. First, we observe the following: $ \lambda_{H_i'}:={\rm inf}\{t|K_{\mcF'_i}+\Theta_i'+H_i'+\M_{X_i'}+tA'_i$ is nef over $X\} = {\rm inf} \{t|K_{\mcF_i'}+\Theta_i'+\M_{X_i'}+tA'_i$ is nef over $X\}=: \lambda_i$. This  can be seen as follows: since $H_i'$ is nef, clearly $\lambda_{H'_i} \leq \lambda_i$. Suppose $\lambda_{H'_i} < \lambda_i$. Then $K_{\mcF_i'}+\Theta_i'+\M_{X_i'}+\lambda_{H_i'}A'_i$ is not nef over $X$. So there exists an extremal ray $R_i \subset \overline{NE}(X_i'/X)$ such that $(K_{\mcF_i'}+\Theta_i'+\M_{X_i'}+\lambda_{H_i'}A'_i) \cdot R_i <0$. Since $(K_{\mcF_i'}+\Theta_i'+\M_{X_i'}+\lambda_i A'_i) \cdot R_i \geq 0$ and $\lambda_i > \lambda_{H_i'}$, it follows that $A_i' \cdot R_i >0$. This implies $(K_{\mcF_i'}+\Theta_i'+\M_{X_i'}) \cdot R_i <0$. Thus $R_i$ is spanned by a rational curve tangent to $\mcF_i'$. Thus, $H_i' \cdot R_i =0$ which would force $R_i$ to be $K_{\mcF_i'}+\Theta_i'+H_i'+\M_{X_i'}+\lambda_{H_i'}A'_i$-negative, a contradiction. Thus $\lambda_{H_i'}= \lambda_i$.\\

We use  the arguments of \cite[Theorem 2.3]{SSMMP} to conclude. Let $X_1':=X' \dashrightarrow X_2' \dashrightarrow $ be the aforementioned MMP over $X$. Choose $N \gg 0$ such that $X_i' \dashrightarrow X_{i+1}'$ is a flip for all $i \geq N$. For each $i \geq N$, let $G_i$ be an ample divisor on $X_i'$ such that if $G_{iN}$ denotes its strict transform on $X_N'$, then $G_{iN} \rightarrow 0$ in $N^1(X_N'/X)$ (for instance, we may take $G_i:=\frac{1}{i}H_i$, where $H_i$ is reduced and ample). Since $K_{\mcF_i'}+\Theta_i'+\M_{X_i'}+\lambda_iA_i+G_i$ is ample over $X$, the relative stable base locus $B((K_{\mcF_N'}+\Theta_N'+\M_{X_N'}+\lambda_iA_N +G_{iN})/X)$ has codimension at least $2$. Letting $ i \rightarrow \infty $, it follows that there exists $W \subset X_N'$ of codimension at least $2$ such that if $C \subset X_N'$ is a curve with $(K_{\mcF_N'}+\Theta_N'+\M_{X_N'})  \cdot C <0$ and $C$ maps to a point on $X$, then $C \subset W$. Now, $K_{\mcF_N'}+\Theta_N'+\M_{X_N'} \equiv_X E_{1N}-E_{2N}$, where $E_{1N}$, $E_{2N}$ are effective and exceptional over $X$. Then by the negativity lemma of \cite[Lemma 3.3]{Bi}, it follows that $E_{2N} \geq E_{1N}$ (note: here we assume $X_N' \to X$ is not small; if the morphism is small then of course we are done). Since $E_1$ and $E_2$ had no components in common, this implies $E_{1N}=0$ as required. This ends the proof of Claim \ref{clm:dlt-model}. 
\end{proof}

Thus we have 
 \begin{equation} \label{pullback}
    \beta''^{*}(K_{\mcF}+\Delta+\mu S + \mathbf{M}_{X''})= K_{\mcF''}+ \Theta''+F_\mu + \mathbf{M}_{X''}    
\end{equation}

Now since $R$ is contained in $W_0$ and $\beta''^{-1}W_0$ is a divisor on $X''$, there is a $K_{\mcF''}+\Theta''+F_\mu+\bfM_{X''}$-negative extremal ray, say $R''$ contained in an irreducible component of $S''+ \sum E_i''+F_\mu$ such that $\beta''_*R''=R$. We claim that $R''$ is not contained in any component of $F_\mu$. Indeed, if $R''\subset F_1$ for some component $F_1$ of $F_\mu$, then $R\subset W_0\cap\beta''(F_1)$. Choose $0\<\lambda<\mu$ such that $\beta''(F_1)$ is a glc center of $(\mcF, \Delta+\lambda S, \bfM)$; note that such a $\lambda$ exists here since $(\mcF, \Delta, \bfM)$ is glc. Then there is a glc center $V$ of $(\mcF, \Delta+\lambda S, \mbf M)$ contained in $W_0\cap \beta''(F_1)$ such that $R\subset V$; this contradicts the minimality of $W_0$. Thus $R''$ is contained in a component of $ S''+ \sum E''_i$, say $E''$. We also have $(K_{\mcF''}+\Theta''+\bfM_{X''})\cdot R''<0$, as $R''$ is not contained in any component of $F_\mu$ and thus $F_\mu\cdot R''\>0$. Let $\nu:{ E''}^\nu\to E''$ be the normalization. Since $\mu=1-\mult_{S}(\Delta)$, by adjunction we have $\nu^*(K_{\mcF''}+\Theta''_{\mu}+\mathbf{M})= K_{\mathcal{G}}+\Theta''_{{E''}^\nu}+\mathbf{N}$
where $\mcG$ is the induced foliation on $E''^\nu$. Note that $(\mcG, \Theta''_{{E''}^\nu}, \mathbf N)$ is glc and $(K_{\mathcal{G}}+\Theta''_{{E''}^\nu}+\mathbf{N}_{E''^\nu})\cdot R''<0$. Then by cone theorem in dimension $n-1$ there is a rational curve $\xi \in R''$ tangent to $\mathcal{G}$ and satisfying
$0<-(K_{\mathcal{G}}+\Theta''_{{E''}^\nu}+\mathbf{N}_{E''^\nu})\cdot\xi''\<2(n-1)$.\\

Then $\xi:=\beta''(\xi'')\subset X$ is a rational curve generating the extremal ray $R$, i.e. $R=\mbR_+[\xi]$, and $\xi$ is tangent to $\mcF$ (see \cite[Lemma 3.3.(1)]{ACSS}). Thus from eqn. \eqref{pullback} along with the projection formula and the fact that $S\cdot R<0$,  we get that 
$0<-(K_{\mcF}+\Delta+\bfM_{X})\cdot \xi\leq 2n.$
This completes the proof of Claim \ref{clm:existence-of-curve}.
\end{proof}

Now let $\{R_i\}_{i\in I}$ be the set of all $K_{\mcF}+\Delta+\bfM_X$-negative exposed extremal rays of $\NE(X)$ generated by rational curves $\{\xi_{i}\}_{i\in I}$ tangent to the foliation $\mcF$ and satisfy $0<-(K_{\mcF}+\Delta+\bfM_X)\cdot \xi_i \<2\dim X$ for all $i\in I$. Suppose that $V=\NE(X)_{(K_{\mcF}+\Delta+\bfM_X)\>0}+\sum_{i\in I}\mbR_+[\xi_i]$. Now recall the following property of convex cones:  if $K\subset \mbR^N$ is a strongly convex closed cone, and \emph{exp(K)} is the set of all exposed extremal rays, then $\overline{{\rm cone}({\rm exp}(K))}=K$. Thus it's enough to show that $\overline V=V$, i.e. $V$ is a closed cone, however, this follows from a standard argument, for example, see the proof of \cite[Theorem 6.4, Third step, page 152]{Deb01}.

\end{proof}

\begin{corollary}\label{fin} Let $(\mcF,B,\M)$ be an F-glc pair on a normal projective variety $X$, $A$ and $H$ ample $\mbQ$-divisors on $X$ such that $K_\mcF+B+\M+A$ is not nef. Then there exist finitely many $K_\mcF+B+\M+A$-negative (exposed) extremal rays in $\overline{NE}(X)$. Let $\lambda:=$inf $\{t \geq 0|K_\mcF+B+\M+A+tH$ is nef$\}$. Then there exists a $K_\mcF +B+\M+A$-negative extremal ray $R$ such that $(K_\mcF+B+\M+A+\lambda H)\cdot R =0$.

\begin{proof}
    The finiteness claim follows from cone theorem for $(\mcF,B,\M)$. See the proof of \cite[Lemma 9.2]{CS} for the remaining arguments.
\end{proof}



\end{corollary}

\subsection{Canonical bundle formula}

A canonical bundle formula for generalized foliated pairs seems to be unknown at present \footnote{although see \cite{CHLX} for a recent development}. A particular case appeared in \cite[Proposition 6.4]{Liu} (see also \cite[Lemma 9.1]{CS} for a version in dimension three). For our needs, the following makeshift version suffices.

\begin{proposition}\label{cbf} Let $Z$ be a $\mbQ$-factorial normal projective variety and $f: X \to Y $ a projective contraction over $Z$ between normal projective varieties, $\mcF$ an algebraically integrable foliation on $X$ induced by a projective morphism $X \to Z$ and $\mcH$ a foliation on $Y$ induced by $Y \to Z$ such that $\mcF= f^{-1}\mcH$. Now, let $(\mcF, \Delta, \M)$ be a F-glc pair on $X$ with $K_\mcF+\Delta+\M \sim _{\mathbb{Q},Y}0$ and $\M$ is b-abundant, $X$ is $\mathbb{Q}$-factorial and $(X, \Delta, \M)$ is gklt. Then there exists a F-glc pair $(\mcH, B_Y, \N)$ on $Y$ such that $K_\mcF+\Delta+\M \sim _{\mathbb{Q}}f^*(K_\mcH+B_Y+\N_Y)$. Assume $Y$ is of klt type. Then (possibly replacing $Y$ with a small $\mbQ$-factorialization) $(Y, B_Y, \N)$ is gklt.


    \begin{proof}

    We start with an observation showing that the discrepancy difference divisor is vertical over the base for algebraically integrable foliations. Suppose $\pi: X' \to X$ is a birational morphism. We have (see \cite[Lemma 3.1]{Spi}), $\pi^*K_\mcF-K_{\mcF'}=\pi^*K_X-K_{X'}+ \Theta$, where $\Theta:= \mathcal{N}^*_{\mcF'} -\pi^*\mathcal{N}_{\mcF}^*\geq 0$. Here $\mathcal{N}_{\mcF}:=K_\mcF-K_X$. For generalized pairs, this translates to 
    \begin{center}
        $\pi^*(K_\mcF+\Delta+\M_X)-(K_{\mcF'}+\M_{X'}) = \pi^*(K_X+\Delta+\M_X)-(K_{X'}+\M_{X'})+\Theta$.
    \end{center}
    Now if $\mcF$ is given by $g:X \to Z$, then $K_\mcF=K_{X/Z}-R_g+E$, where $R_g$ is the ramification divisor and $E$ is some $g$-exceptional divisor. Thus $\mathcal{N}_\mcF= E-R_g-f^*K_Z$ is vertical over $Z$ and so is $\Theta$. In particular, the discrepancy difference divisor is vertical over $Z$.\\

    Consider 
    \begin{center}
    
        \begin{tikzcd}
X'' \arrow[r,"\nu"] \arrow[d, "f''"] & X^{'} \arrow[r, "\mu"] \arrow[d, " f'"] & X \arrow[d,"f"]\\
Y'' \arrow[r, "\sigma"] &Y^{'} \arrow[r, "\pi"] &Y
\end{tikzcd}
\end{center}

where $f': (X', \Sigma_{X'}) \to (Y', \Sigma_{Y'})$ is an equidimensional toroidal reduction of $f$ with $\mu_*^{-1}\Delta +\Ex \mu \subset \Sigma_{X'}$ such that $\M$ descends to $X'$ as some pullback of a big and nef divisor, $\sigma : Y'' \to Y'$ is a finite morphism (see \cite[Proposition 5.1]{AK}, \cite[Theorem B.7]{Hu}) and $f''$ is a weakly semistable base change of $f'$. In particular, letting $\Sigma_{X''}$ denote the toroidal divisor on $X''$, we can arrange that the inverse image of $\Delta$, $\Ex \mu$ and $R_\nu$ are all part of $\Sigma_{X''}$. \\

Let $\Delta'$ be defined by $K_{\mcF'}+\Delta'+\M_{X'}= \mu^*(K_\mcF+\Delta+\M_X)$. Letting $R':= \Sigma_{\epsilon(P)=0}(f'^*P-f'^{-1}P)$, we have $K_{\mcF'/\mcH'}+\Delta'+\M_{X'}=K_{X'/Y'}+\Delta'-R'+\M_{X'} \sim_{\mathbb{Q}, Y'}0$ (see \cite[2.9]{Dru17}). Let $K_{X'}+\Delta'-R'+\M_{X'} \sim f'^*(K_{Y'}+B_{Y'}+\N_{Y'})$, where $B_{Y'}$ is the discriminant induced by $(X', \Delta'-R')$ .We will show that $\N$ is b-nef.\\

Let $U \subset Y$ be any large open subset such that $f|_{f^{-1}(U)}$ is equidimensional and let $X_U:=f^{-1}(U)$. Then we have $(K_{\mcF/\mcH}+\Delta+\M)|_{X_U}= (K_{X/Y}+\Delta-R+\M)|_{X_U} \sim_U 0$, where $R=\Sigma_{\epsilon(P)=0}(f^*P-f^{-1}P)$. Let $B_U$ be the discriminant on $U$ induced by $(X, \Delta-R,\M)$ and $B_Y$ its Zariski closure. Then it is easy to check that $B_Y \geq 0$. Let $\N_Y$ be defined by the requirement $K_\mcF+\Delta+\M_X \sim_{\mbQ}f^*(K_\mcH + B_Y+ \N_Y)$. Possibly, shrinking $U$ further to another large open subset, we can arrange that $\mu^*((K_{\mcF/\mcH}+\Delta+\M)|_{X_U})=(K_{\mcF'/\mcH'}+\Delta'+\M)|_{X'_U}$, where $X'_U:=\mu^{-1}X_U$. Then we also have $(K_{X'/Y'}+\Delta'-R'+\M)|_{X'_U}= \mu^*((K_{X/Y}+\Delta-R+\M)|_{X_U})$. This implies that $\pi_*B_{Y'}=B_Y \geq 0$. 
Let $B':=\Delta'-R'$. Let $\M_{X''}:=\nu^*\M_{X'}$ and $B''$ be defined  by $\nu^*(K_{X'}+B'+\M_{X'})=K_{X''}+B''+\M_{X''}$. Now, if $B_{Y''}$ is the discriminant induced by $(X'', B'')$, then $\sigma^*(K_{Y'}+B_{Y'})=K_{Y''}+B_{Y''}$ (see \cite[Theorem 3.2]{Amb99}). We conclude that $\sigma^*\N_{Y'}=\N_{Y''}$. Let $E \subset X'$ be a prime divisor such that coeff$_E\Delta'=1$. Since $(\mcF', \Delta')$ is sub-lc, this forces $\epsilon(E)=1$. Since $E$ is a component of the discrepancy difference divisor, it is forced to be vertical over the base of $\mcF'$, i.e, $\epsilon (E)=0$. Thus coeff$_E\Delta' <1$, implying that $(X',B')$ is sub-klt.\\

We have $K_{X''}+B''+\M_{X''}\sim f''^*(K_{Y''}+B_{Y''}+\N_{Y''})$. We claim that $\N$ descends to a nef divisor on $Y''$. Indeed, we can write $\M_{X''}\sim A_n+\frac{1}{n}E$ for $n \gg 0$ with $A_n$ semiample and $E \geq 0$. Since $(X'', B'')$ is sub-klt, so is $(X'', B''+\frac{1}{n}E)$ for $n \gg 0$. Consider the modified canonical bundle formula $K_{X''}+(B''+\frac{1}{n}E)+(\M_{X''}- \frac{1}{n}E) \sim f''^*(K_{Y''}+B_{Y''}^n+\N_{Y''}^n)$. By \cite[Theorem 4.13, 4.15]{Fil}, $\N_{Y''}^n$ descends to a nef divisor on $Y''$ for all $n$. We have $B_{Y''}^n \to B_{Y''}$ as $n \to \infty$, implying that $\N_{Y''}^n \to \N_{Y''}$ numerically. Thus $\N_{Y''}$ is nef. But the same argument applies to any higher model of $f'':X'' \to Y''$. In particular, $\N_{Y'''}$ is nef for all higher models $\pi: Y''' \to Y''$. By negativity lemma, $\N_{Y'''}=\pi^*\N_{Y''}-E$ for some $E \geq 0$ which is $\pi$-exceptional. Now, since $\pi^*\N_{Y''}^n=\N_{Y'''}^n \to \N_{Y'''}$ numerically, we have $\N_{Y'''} \equiv_{\pi}0$. Thus $E \equiv _\pi 0$. This can only happen if $E=0$. Thus $\N_{Y'''}=\pi^*\N_{Y''}$. In other words, $\N$ descends to a nef divisor on $Y''$. It follows that $\N_{Y'}^n$  and $\N_{Y'}$ descend to nef divisors on $Y'$ (this follows from the corresponding fact for $Y''$ and can be checked on higher models of $Y'$ by taking fiber products).\\

$(Y', B_{Y'},\N)$ is clearly sub-gklt. This shows that $(Y, B_Y,\N)$ is gklt (possibly replacing $Y$ with a small $\mbQ$-factorialization). Indeed, let $\Delta_{Y'}:= \pi^*(K_Y+B_Y+\N_Y)-(K_{Y'}+\N_{Y'})$. We have 
\begin{center}
 $\pi^*(K_\mcH+B_Y+\N_Y) -(K_{\mcH'}+\N_{Y'})= \pi^*(K_Y+B_Y+\N_Y)-(K_{Y'}+\N_{Y'})+\Xi$    
\end{center}
for some $\Xi \geq 0$. This implies that $B_{Y'}= \Delta_{Y'}+\Xi$. Since $(Y', B_{Y'},\N)$ is sub-gklt, so is $(Y', \Delta_{Y'},\N)$. This shows that $(Y, B_Y, \N)$ is gklt.
We are left to show that $(\mcH', B_{Y'}, \N)$ is sub-glc. Since we can always replace $f'$ with a higher equidimensional toroidal reduction, it suffices to just check the coefficients of $B_{Y'}$. For $P$ a prime divisor on $Y'$, let $a_P:=$sup $\{t|(X', \Delta'-R'+tf'^*P)$ is sub-lc over $\eta_P\}$. Then $B_{Y'}= \Sigma_P(1-a_P)$. We show that $1-a_P \leq \epsilon(P)$ for all $P$. Since clearly $1-a_P \leq 1$, the $\epsilon(P)=1$ case is clear. Let $\epsilon(P)=0$, $f'^{-1}P = \Sigma_i{Q_i}$ and coeff$_{Q_i}\Delta'= \lambda_i \leq 0$ for all $i$. It follows from the definition of $B_{Y'}$ that $\lambda_i-(l(Q_i)-1)+a_P \cdot l(Q_i) \leq 1$ for all $i$, where $l(Q_i)$ denotes the ramification index of $Q_i$. Moreover equality holds in the above inequality for some $i=i_0$. Indeed, let $f:(X, \Sigma_X) \to (Y, \Sigma_Y)$ be toroidal with $(Y, \Sigma_Y)$ log smooth. If $P\subset Y$ is a component of $\Sigma_Y$, then the claim is clear, so suppose not. Then $(Y, \Sigma_Y+P)$ is log smooth at $\eta_P$. Since toroidal morphisms have Property ($*$) (def \ref{prop*}, see \cite[Prop 2.16]{ACSS}), it follows that $(X, \Sigma_X^h+f'^{-1}P)$ is lc at $f'^{-1}(\eta_P)$. Thus $\lambda_{i_0}=(1-a_P)\cdot l(Q_{i_0})$, thereby showing that $1-a_P \leq 0$ if $\epsilon(P)=0$. This concludes the proof.

    \end{proof}
\end{proposition}

\section{Basepoint free theorem}

\subsection{Basepoint freeness for foliations induced by nice morphisms}
We are in a position to prove \cite[Conjecture 4.1]{CS3} for foliations induced by nice morphisms. 

\begin{lemma}\label{ab}
Let $\mcF$ be an algebraically integrable foliation on a normal projective variety $X$ induced by a projective contraction $f:X \to Y$ to a $\mathbb{Q}$-factorial normal projective variety without exceptional divisors. Let $(\mcF,B,\M)$ be a F-glc pair on $X$ such that letting $G:=f^{-1}B_f$, where $B_f$ is the reduced branch divisor of $f$, there exists $\epsilon >0$ such that $(X, B+G-\epsilon f^*B_f, \M)$ is gklt. Let $A$ be an ample divisor on $X$ such that $K_\mcF+A+B+\M_X$ is nef. Then $K_\mcF+A+B+\M_X$ is semiample.

\begin{proof} We argue along the lines of \cite[Lemma 9.3]{CS}. Let $U \subset Y$ be a large open subset such that $f|_{f^{-1}(U)}$ is equidimensional. Let $X_U:=f^{-1}(U)$ and $G_U$ denote the inverse image of the branch divisor of $f|_{X_U}$. Note that $X_U \subset X$ is large open. Now $K_\mcF|_{X_U} \sim _U (K_X+G_U)|_{X_U}$. Since $X_U$ and $U$ are both large open, this linear equivalence extends to $K_\mcF \sim_{\mbQ,Y} K_X+G$, where $G$ is the Zariski closure of $G_U$ in $X$ and by definition is the inverse image of the branch divisor of $f$.\\

Let $\lambda:=$inf $\{t\;|\;K_\mcF+B+\M+\frac{1}{2}A+tA$ is nef$\}$. We may assume $\lambda=\frac{1}{2}$, since otherwise $K_\mcF+A+B+\M_X$ is ample and we are done. By cone theorem for $(\mcF, B,\M)$ (see Theorem \ref{cone} and Corollary \ref{fin}), there exists an extremal ray $R \subset \overline{NE}(X)$ such that $(K_\mcF+B+\M+\frac{1}{2}A) \cdot R<0$, $(K_\mcF+B+\M+A)\cdot R =0$ and $R$ is spanned by a rational curve contracted by $f$. Now observe that 
\begin{center}
    $K_\mcF+B+\M+\frac{1}{2}A \sim_Y K_X+B+G+\M+\frac{1}{2}A \sim_Y K_X+B+G-\epsilon f^*B_f +\M+\frac{1}{2}A$. 
\end{center}
Thus by contraction theorem for klt pairs, there exists a projective contraction $\phi_1:X \to X_1$ over $Y$ which contracts $R$. Let $f_1:X_1 \to Y$ denote the induced morphism. Assume $\phi_1$ is birational. Let $H_1$ be an ample divisor on $X_1$ such that $A-\phi_1^*H_1 \sim \Theta$ is ample. Then $(\mcF,B, \M+\T+ \phi_1^*\H_1)$ is F-glc and hence so is the transformed pair $(\mcF_1, B_1, \M+\T+ \H_1)$. Let $\lambda_1:=$inf $\{t|K_{\mcF_1}+B_1+\M_{X_1}+\T_{X_1}+\frac{1}{2}H_1+tH_1$ is nef$\}$. If $\lambda_1<\frac{1}{2}$, then $K_\mcF+B+\M+A$, being the pullback of an ample divisor, is semiample. So we may assume $\lambda_1= \frac{1}{2}$. By cone theorem for $(\mcF_1, B_1, \M+\T)$, there exists an extremal ray $R_1 \subset \NE(X_1)$ such that $(K_{\mcF_1}+B_1+\M_{X_1}+\T_{X_1}+\frac{1}{2}H_1)\cdot R_1 <0$ and $(K_{\mcF_1}+B_1+\M_{X_1}+\T_{X_1}+H_1)\cdot R_1=0$. Note that 

\begin{center}
   $K_{\mcF_1}+B_1+\M_{X_1}+\T_{X_1}+H_1 \sim_Y K_{X_1}+B_1+G_1-\epsilon f_1^*B_{f_1}+\M_{X_1}+\T_{X_1}+H_1$ 
\end{center}
where the RHS is a gklt pair (since $(X, B+G-\epsilon f^*B_f, \M+\T+\phi_1^*\H_1)$ is gklt). In particular, $(X_1, B_1, \M+\T+\H_1)$ is gklt (possibly replacing $X_1$ with a small $\mbQ$-factorialization). We can thus continue this process $X \xrightarrow{\phi_1}X_1 \xrightarrow{\phi_2}X_2 \cdots$ as long as $\phi_i$ is birational eventually ending up with two possible outcomes: \vspace{2 mm}

\begin{enumerate}
    \item An F-glc pair $(\mcF_n, B_n, \M'+A_n)$ on $X_n$ with $A_n$ ample such that if $\lambda_n:=$inf $\{t|K_{\mcF_n}+B_n+\M'_{X_n}+\frac{1}{2}A_n+tA_n\}$ is nef$\}$, $\lambda_n<\frac{1}{2}$. Then $K_\mcF+B+\M_X+A$ being the pullback of ample, is semiample.\\

    \item An F-glc pair $(\mcF_n,B_n, \M'+\A_n)$ on $X_n$ with $A_n$ ample such that $\lambda_n=\frac{1}{2}$ and the contraction over $Y$ defined by the above prescription is of Fano type. For notational simplicity, we drop all subscripts. Thus we are in the following situation: there exists a projective contraction over $Y$, say $g:X \to Z$ such that $K_\mcF+B+\M'+\A \sim_Z 0$. Replacing $X$ with a small $\mbQ$-factorialization, we may assume that $(X, B, \M'+\A)$ is gklt. Then by the Theorem \ref{cone} and Proposition \ref{cbf}, there exists a foliation $\mcH$ on $Z$ induced by $Z \to Y$ with $\mcF=g^{-1}\mcH$ such that we can write $K_\mcF+B+\M'+\A \sim g^*(K_\mcH+B_Z+\N_Z)$, where $(\mcH, B_Z,\N)$ is F-glc and $(Z, B_Z, \N)$ is gklt (possibly replacing $Z$ with a small $\mbQ$-factorialization). Let $H_Z$ be an ample divisor on $Z$ such that $A-g^*H_Z$ is ample. Then by Proposition \ref{cbf}, we get an F-glc pair $(\mcH, B_Z, \N')$ such that $K_\mcF+B+\M'+A-g^*H_Z \sim g^*(K_\mcH+B_Z+\N'_Z)$. Comparing with the above canonical bundle formula gives $\N=\N'+H_Z$. Furthermore, since $K_X+B+G -\epsilon f^*B_f+\M'+A \sim _{\mbQ,Z}0$, it follows that $K_{\mcH}+B_Z+\N_Z$ is linearly equivalent over $Y$ to a gklt pair by \cite{Fil}. We can thus continue the process by replacing $X$ with $Z$. Due to drop in Picard number, it eventually ends.
\end{enumerate}

    \end{proof}
\end{lemma}

\subsection{Basepoint freeness for F-dlt pairs} First we recall the definition of F-dlt for algebraically integrable foliations:

\begin{definition}\label{dlt}(F-dlt)\cite[Section 3.2]{ACSS} Let $(\mcF, \Delta)$ be a foliated log canonical pair on a normal projective variety $X$ where $\mcF$ is defined by a dominant rational map $f: X \dashrightarrow Y$. $(\mcF, \Delta)$ is called \emph{F-dlt} if there exist birational morphisms $\pi: X' \to X$ and $Y' \to Y$, reduced divisors $\Sigma_{X'}$ on $X'$ and $\Sigma_{Y'}$ on $Y'$ with $\pi_*^{-1}\Delta + \Ex \pi \subset \Sigma_{X'}$ such that there exists an induced morphism $f':(X', \Sigma_{X'}) \to (Y', \Sigma_{Y'})$ which is equidimensional and toroidal inducing the foliation $\mcF':=\pi^{-1}\mcF$ and $(Y', \Sigma_{Y'})$ is log smooth (called an equidimensional toroidal reduction from now on) satisfying the following property:
\newline
for any $\pi$-exceptional prime divisor $E$ on $X'$, 
writing $K_{\mcF'}+ \Delta' = \pi^*(K_\mcF +\Delta)$, we have coeff$_E \Delta'<\epsilon(E)$. \footnote{At present, it is not clear to us if with this definition, any foliated lc pair will have an F-dlt modification.}
\end{definition}

\begin{theorem}\label{CSconj}\cite[Conjecture 4.1]{CS3}
    Let $\mcF$ be an algebraically integrable foliation on a $\mathbb{Q}$-factorial normal projective variety $X$ which is induced by a dominant rational map $f:X \dashrightarrow Y$. Suppose $(\mcF, B)$ is a F-dlt pair on $X$ such that $(X,B)$ is klt and $A$ an ample $\mathbb{Q}$-divisor on $X$. If $K_{\mcF}+A+B$ is nef, then it is semiample.

\begin{proof} We use some ideas from the proof of \cite[Theorem 10.3]{CS}. Let $\Delta:=A+B$. Consider an equidimensional toroidal reduction $f':(X',\Sigma_{X'}) \to (Y', \Sigma_{Y'})$ associated to $(\mcF,B)$ as in Definition \ref{dlt}. We can write $K_{\mcF'}+B_{X'}= \pi^*(K_\mcF+B)+E$, where $B_{X'} \geq 0$ and $E \geq 0$ is $\pi$-exceptional. Note that if $E_i$ is a $\pi$-exceptional divisor such that $\epsilon(E_i)=0$, then $E_i$ is a component of $E$. Replacing $A$ by a general member in its $\mathbb{Q}$-linear system, we may assume that $\pi^*A = \pi_*^{-1}A$. Let $\Delta_{X'}:=B_{X'}+\pi^*A$. Then $K_{\mcF'}+\Delta_{X'}=\pi^*(K_\mcF+\Delta)+E$. We can choose $C \geq 0$ such that $-C$ is $\pi$-ample and $\pi^*A -\delta C$ is ample for all $0< \delta \leq 1$. Let $\epsilon>0$ and $\Gamma':=\Delta_{X'}-\delta C +\epsilon \Sigma E_i$ where we sum over all $\pi$-exceptional $\mcF'$-non-invariant divisors. Then 
    \begin{center}
    $K_{\mcF'}+\Gamma'=K_{\mcF'}+\Delta_{X'}-\delta C +\epsilon \Sigma E_i$.
    \end{center}
    Now $ E+ \epsilon \Sigma E_i \geq 0$ and contains all $\pi$-exceptional divisors. Thus $F':=E-\delta C +\epsilon \Sigma E_i \geq 0$ and contains all $\pi$-exceptional divisors for suitable choices of $\epsilon$ and $\delta$. Note that $\Delta_{X'}= \pi_*^{-1}\Delta+E'$ for some $E' \geq 0$ which is $\pi$-exceptional and every component of $E'$ is $\mcF'$-non-invariant. Also, with this, we have $\Gamma'= \pi_*^{-1}\Delta +E' -\delta C + \epsilon \Sigma E_i = \pi^*A -\delta C +\pi_*^{-1}B+E'+ \epsilon \Sigma E_i$. By \cite[Lemma 3.1]{ACSS}, we have $(\mcF', \Sigma_{X'}^h)$ is lc and has property ($*$) \cite[Proposition 2.16]{ACSS}. Thus $(\mcF', \pi_*^{-1}B +E' +\epsilon \Sigma E_i)$ is lc and has property ($*$). In particular, MMP with respect to this pair preserves equidimensionality and property ($*$) \cite[Proposition 2.18]{ACSS}. Note that $(\mcF', \pi_*^{-1}B+E'+\epsilon \Sigma E_i, \H)$ is F-glc for $H \in |\pi^*A-\delta C|_\mbQ$. Let $\Gamma := H+\pi_*^{-1}B+E'+ \epsilon \Sigma E_i \sim \Gamma'$. Then $K_{\mcF'}+ \Gamma \sim \pi^*(K_\mcF+\Delta) +F'$. Furthermore, $K_{\mcF'}+\pi_*^{-1}B+E'+ \epsilon \Sigma E_i +\H \sim_{\mbQ,Y'}K_{X'}+\pi_*^{-1}B+E'+ \epsilon \Sigma E_i +G-\alpha f'^*B_{f'}+\H$
    and the latter is a gklt pair for some $\alpha >0$.
    Here $G'$ denotes the inverse image of the branch divisor of $f'$ and $B_{f'}$ is the reduced branch divisor of $f'$. In particular, we can run a $K_{\mcF'}+\Gamma$-MMP with scaling of $H$, which is a $K_{X'}+\Gamma +G'-\alpha f'^*B_{f'}$-MMP over $Y'$ with scaling of $H$, as in the proof of Claim \ref{clm:dlt-model}, say $ \phi: X' \dashrightarrow X''/Y'$ such that $K_{\mcF''}+\Gamma''$ is nef, where $ \Gamma'':=\phi_* \Gamma$. Let $f'': X'' \to Y'$ be the induced morphism and note that $f''$ is equidimensional.
    Further, by arguing as in the proof of Corollary \ref{gmm} below, we get an F-glc pair $(\mcF'', B''+E''+\epsilon \Sigma E''_i, \N+\A)$ on $X''$ with $K_{\mcF''}+B''+E''+ \epsilon \Sigma E_i''+\N_{X''}+A \sim K_{\mcF''}+\Gamma''$, where $A$ is ample and such that $(X'', B''+E''+\epsilon \Sigma E''_i + G''-\alpha f''^*B_{f''}, \N+\A)$ is gklt, where $G''$ denotes the  inverse image of the branch divisor and $B_{f''}$ the reduced branch divisor of $f''$. Let $\Gamma'':=\phi_*\Gamma$. Then $K_{\mcF''}+\Gamma''$ is semiample by Lemma \ref{ab}. \\

    Now let $X' \xleftarrow{p} W \xrightarrow{r} X''$ be a common resolution of $\phi$ and $X \dashrightarrow X''$ with $q: W \to X$ the induced morphism. Then
    \begin{center}
        $p^*(K_{\mcF'}+ \Gamma) \sim q^*(K_{\mcF}+\Delta)+ p^*F'$, 
        
        \end{center}
        where $p^*F' \geq 0$ is $q$-exceptional and 
        \begin{center}
            $p^*(K_{\mcF'}+\Gamma) \sim r^*(K_{\mcF''}+\Gamma'')+F$
        \end{center}
        for some $F\geq 0$ which is $r$-exceptional. Putting these together, we have
        \begin{center}
            $q^*(K_\mcF+ \Delta)+p^*F' -F \sim r^*(K_{\mcF''}+\Gamma'')$.
        \end{center}
        By nefness of $K_{\mcF''}+\Gamma''$, negativity lemma over $X$ gives $F \geq p^*F'$. Similarly, nefness of $K_{\mcF}+\Delta$ and negativity over $X''$ gives $p^*F' \geq F$. We conclude $q^*(K_{\mcF}+\Delta) \sim r^*(K_{\mcF''}+ \Gamma'')$. This finishes the proof.

\end{proof}
\end{theorem}

\section{Applications}
The arguments used to prove the above Theorem can be used to show that abundance holds in a more general situation. We record this below as a corollary of Theorem \ref{CSconj}. First we need to recall some definitions from Nakayama's book \cite{Nak04}.\\

For a pseudoeffective divisor $D$ on a smooth projective variety $X$, let $\kappa(D)$ and $\nu(D)$ denote its Iitaka and numerical dimension \cite [Definition 2.10]{GL13} respectively. Then there exists a decomposition of $D$ into $\mathbb{R}$-divisors, namely $D=P_\sigma(D)+N_\sigma(D)$. $P_\sigma(D)$ and $N_\sigma(D)$ are called the positive and negative parts of $D$ respectively. We have $\nu(D)= \nu(P_\sigma(D))$ and $\kappa(D)=\kappa(P_\sigma(D))$ \cite[Chapter 5, Proposition 2.7]{Nak04}, \cite[Lemma 2.9]{GL13}. Moreover, if $\mu:W \to X$ is a birational morphism from a smooth projective variety and $D$ is a pseudoeffective divisor on $X$, then for any effective $\mu$-exceptional divisor $E$ on $W$, we have $P_\sigma(\mu^*D)=P_\sigma(\mu^*D+E)$ \cite[Lemma 2.16]{GL13}. Recall that a pseudoeffective divisor $D$ is called \emph{abundant} if its numerical and Iitaka dimensions match. With this, we have the following result on abundance of $K_\mcF+\Delta$ when $\Delta$ contains an ample divisor.

\begin{corollary} (abundance)
    Let $(\mcF, B)$ be a F-dlt pair on a $\mathbb{Q}$-factorial normal projective variety $X$ such that $(X,B)$ is klt. Let $A$ be an ample $\mathbb{Q}$-divisor on $X$ and $\Delta:=A+B$. Suppose $K_\mcF+\Delta$ is pseudoeffective and $\mcF$ is algebraically integrable. Then there exists a birational morphism $q:W \to X$ from a smooth projective variety such that $P_\sigma(q^*(K_\mcF+\Delta))$ is semiample. In particular, $K_\mcF+\Delta$ is abundant.

    \begin{proof}
        We follow the notation from the proof of the above Theorem. Choose $\pi: X' \to X$, $X' \xleftarrow{p}W \xrightarrow{r}X''$ as in the proof of Theorem \ref{CSconj}. As observed above, we have 

        \begin{center}
            $p^*(K_{\mcF'}+\Gamma) \sim q^*(K_\mcF+\Delta)+p^*F'$ where $p^*F' \geq 0$ is $q$-exceptional.
        \end{center}
        This gives $P_\sigma(q^*(K_\mcF+\Delta))=P_\sigma(p^*(K_{\mcF'}+\Gamma))$. We also have
        \begin{center}
            $p^*(K_{\mcF'}+\Gamma)=r^*(K_{\mcF''}+\Gamma'')+F$
        \end{center}
        for some $F \geq 0$ which is $r$-exceptional and $K_{\mcF''}+\Gamma''$ is semiample. Thus $P_\sigma(p^*(K_{\mcF'}+\Gamma)= P_\sigma(r^*(K_{\mcF''}+\Gamma''))$ is semiample. We thus conclude that $P_\sigma(q^*(K_\mcF + \Delta))$ is semiample. Since for any pseudoeffective divisor $D$, $P_\sigma(D)$ and $D$ have the same numerical and Iitaka dimension, the abundance of $K_\mcF+\Delta$ follows.

    \end{proof}
\end{corollary}

\begin{corollary}(Contraction theorem) \label{cont} Let $(\mcF,B)$ be an F-dlt pair on a $\mathbb{Q}$-factorial normal projective variety $X$ such that $(X,B)$ is klt and $R \subset \overline{NE}(X)$ a $K_\mcF+B$-negative exposed extremal ray. Then there exists a projective contraction $\phi:X \to Y$ satisfying the  property that for any integral curve $C \subset X$, $\phi(C)$ is a point iff $[C] \in R$.

    \begin{proof}
    Let $H_R :=K_{\mcF}+B+A$ be a supporting Cartier divisor of $R$, where $A$ is some ample $\mathbb{Q}$-divisor on $X$. Then $H_R$ is nef and hence semiample by Theorem \ref{CSconj}. Let $\phi:X \to Y$ be the Iitaka fibration of $H_R$.  
    \end{proof}
\end{corollary}

\begin{corollary}[Existence of good minimal model] \label{gmm} Let $\mcF$ be an algebraically integrable foliation on a $\mathbb{Q}$-factorial normal projective variety $X$ which is induced by a projective contraction $f:X \to Y$ to a $\mbQ$-factorial normal projective variety without exceptional divisors. Suppose $(\mcF,B,\M)$ is glc and that $(X, B+G-\epsilon f^*B_f, \M)$ is gklt for some $\epsilon >0$, where $G$ denotes the inverse image of the branch divisor and $B_f$ the reduced branch divisor of $f$. Let $A$ be an ample $\mbQ$-divisor on $X$. Then we have the following:
\begin{enumerate}
    \item If $K_\mcF+B+\M+\A$ is pseudoeffective, the F-glc pair $(\mcF, B, \M+\A)$ has a good minimal model.
    \item If $K_\mcF+B+\M$ is not pseudoeffective, then we can run a $K_\mcF+B+\M$-MMP terminating with a Mori fiber space.
\end{enumerate}


\begin{proof}

Since $K_\mcF+B+\M+\A \sim_Y K_X+B+G-\epsilon f^*B_f+\M+\A$, by cone theorem for the F-glc pair $(\mcF, B, \M+\A)$, we may run a $K_\mcF+B+\M+\A$-MMP with scaling of $A$ which is a $K_X+D$-MMP (for some klt pair $(X,D)$ with $D$  big) over $Y$ with scaling of $A$ (see proof of Claim \ref{clm:dlt-model}). We claim that the property that the moduli part contains an ample divisor carries through the MMP. Indeed, let $\phi_1:X \dashrightarrow X_1$ be a step of the MMP. Choose an ample divisor $H_1$ on $X_1$ such that if $H_X:=\phi_{1*}^{-1}H_1$, there exists $C \sim A-H_X$ which is ample. $\phi_1$ is clearly also a $K_\mcF+B+\M+\A+\epsilon C$-MMP if $\epsilon >0$ is sufficiently small. Moreover, $(\mcF, B, \M+\A+\epsilon \C)$ is F-glc and so is $(\mcF, B, \M+(1-\epsilon)\A+\epsilon \C)$. Clearly, $\phi_1$ is also a $(K_\mcF+B+\M+(1-\epsilon)\A +\epsilon \C)$-MMP as well. Thus, $(\mcF_1, B_1, \M+(1-\epsilon) \A+\epsilon \C)$ is glc. This implies that $(\mcF_1, B_1, \M+(1-\epsilon)\A+ \epsilon \C +\epsilon H_1)$ is glc. Now note that 
\begin{equation} \label{amplemoduli}
    K_{\mcF_1}+B_1+\M_{X_1}+(1-\epsilon)\A_{X_1}+\epsilon \C_{X_1}+\epsilon H_1 \sim _{\mbQ} K_{\mcF_1}+B_1+\M_{X_1}+\A_{X_1}= 
    \phi_{1*}(K_\mcF+B+\M_X+\A).
\end{equation}
This proves the claim. Let $\phi: X \dashrightarrow X'$ denote the full MMP (note that it terminates by \cite{BCHM}; see proof of Claim \ref{clm:dlt-model} for details of the arguments). Let $(\mcF', B', \N'+H')$ denote the induced F-glc pair on $X'$ with $H'$ ample. If $f': X' \to Y$ is the induced morphism, then $f'$ has no exceptional divisors. Moreover, observe that $(X', B'-\epsilon f'^*B_{f'}, \N'+H')$ is gklt. Thus, $K_{\mcF'}+B'+\N'+H'$ is semiample by Lemma \ref{ab}. \\

Now suppose that $K_\mcF+B+\M$ is not pseudoeffective. Then there exists an ample divisor  $A$ such that $K_\mcF+B+\M+A$
 is not pseudoeffective. We  can run a $K_\mcF+B+\M+A$-MMP as above terminating with a $K_\mcF+B+\M$-Mori fiber space.


\end{proof}
    
\end{corollary}


\begin{definition}[Moduli part \cite{ACSS}]\label{moddef} Let $f: X \to Z$ be a contraction from a normal variety to a $\mathbb{Q}$-factorial variety. Let $\Delta$ be an effective divisor on $X$ such that $K_X+\Delta$ is $\mathbb{Q}$-Cartier. Let 
\begin{center}
    $B_Z^{\Delta} := \Sigma_P(1-a_P)P$ where $P$ ranges over all prime divisors in $Z$
    \end{center}
   where $ a_P:=$sup$\{t| (X, \Delta + tf^*P)$ is sub-lc over the generic point of $P$\}.

Then the moduli part of $f$ corresponding to $(X, \Delta)$ is defined by \begin{center}
    $M_X^{\Delta}:= K_X+\Delta- f^*(K_Z+B_Z^{\Delta})$
\end{center}
    
\end{definition}

\begin{corollary}[Semiampleness of the moduli part]\label{mod}
Let $(X,B)$ be a $\mbQ$-factorial klt pair, $A$ an ample $\mbQ$-divisor on $X$. Suppose $f: (X, \Sigma_X) \to (Y, \Sigma_Y)$ is an equidimensional toroidal morphism to a smooth projective variety such that $K_X+A+B \sim _{\mbQ, f}0$ and components of $B$ are components of the $f$-horizontal part of the toroidal boundary $\Sigma_X$. . Then the moduli part $M_X^{A+B}$ is semiample.

\begin{proof}
Our arguments are similar to those of \cite[Theorem 1.3]{ACSS}. Let $\mcF$ denote the foliation on $X$ induced by $f$. Let $\Delta:= A+B$. Let $G$ denote the inverse image of the branch divisor and $B_f$ the reduced branch divisor of $f$. Note that $(X, B+G)$ has property ($*$) by \cite[Prop 2.16]{ACSS}. Replacing $A$ by a general member in its linear system (note that the moduli part is only determined up to linear equivalence, so is not affected by this), we have $B_Y^{\Delta+G}=B_Y^{B+G}$ (Clearly, $B_Y^{\Delta+G} \geq B_Y^{B+G}$. Since $(X,B+G+a_P^{B+G}f^*P)$ is lc over $\eta_P$, so is $(X,B+A+G+a_P^{B+G}f^*P)$ for any $P$ which is a component of $B_Y^{B+G}$. Thus $a_P^{B+G} \leq a_P^{\Delta+G}$ and hence $B_Y^{B+G} \geq B_Y^{\Delta+G}$). Next we observe that $B_Y^{B+G}=B_f$ (Indeed, if $P \subset Y$ is a prime divisor not contained in $B_f$, then due to Property ($*$), we have coeff$_P(B_Y^{B+G})=0$. If $P \subset B_f$, then coeff$_P(B_Y^{B+G})=1$). Let $\overline{\Delta}:=\Delta+f^*(B_Y^{\Delta+G}-B_Y^{\Delta})$. Clearly, $M_X^{\Delta}=M_X^{\overline{\Delta}}$. Note that $E:=\Delta+G -\overline{\Delta} \geq 0$ and is $f$-very exceptional (i.e. $E$ does not support the fiber over any codimension $1$ point of $Y$). Indeed, for $P \subset B_f$, let $f^*P=n_1Q_1+ \cdots + n_kQ_k$ with $n_1 \geq n_2 \geq \cdots \geq n_k$, then $Q_1 \not\subset$Supp $\Delta+G-\overline{\Delta}$. This can be seen as follows: as a consequence of Property ($*$), it follows that $a_P^{\Delta}= \frac{1}{n_1}$, and thus coeff$_PB_Y^{\Delta}=1-\frac{1}{n_1}$. Since $B_Y^{\Delta+G}=B_f$, it follows that coeff$_P(B_Y^{\Delta+G}-B_Y^{\Delta})= \frac{1}{n_1}$ and thus coeff$_{Q_1}f^*(B_Y^{\Delta+G}-B_Y^{\Delta})=1$. Since $E=G-f^*(B_Y^{\Delta+G}-B_Y^{\Delta})$, it follows that coeff$_{Q_1}E=0$. If $P\not \subset B_f$, then coeff$_PB_Y^{\Delta+G}=$ coeff$_PB_Y^{\Delta}=0$.\\

Since $K_X+\overline{\Delta} \sim _Y 0$, it follows that $K_X+\Delta+G \sim_Y E$. Note that \begin{center}
$K_\mcF+B+A \sim _YK_X+B+G-\epsilon f^*B_f+A$ 
\end{center}
for all $\epsilon>0$. Moreover, $(\mcF,B,\A)$ is F-glc and $(X,B+G-\epsilon f^*B_f,\A)$ is gklt for some $\epsilon>0$ since $f$ is toroidal. Now, by \cite[Theorem 3.4]{Bi}, we can run a $K_\mcF+B+\A$-MMP $\phi:X \dashrightarrow X' $ which is a $K_X+B+G-\epsilon f^*B_f+\A$-MMP over $Y$ with scaling of $A$ such that $\phi_*E=0$ (see the proof of Claim \ref{clm:dlt-model} for details). This is also a $K_\mcF+B$-MMP and thus Property ($*$) and equidimensionality are preserved by it \cite[Prop 2.18]{ACSS}. We have $K_{X'}+\Delta'+G' \sim_Y 0$ and $K_{\mcF'}+\Delta'$ is nef, where $\Delta':=\phi_*\Delta$ and $G':= \phi_*G$. By the arguments used in Corollary \ref{gmm}, we can write $K_{\mcF'}+\Delta' \sim K_{\mcF'}+B'+\N+\H$, where $(\mcF', B', \N+\H)$ is F-glc and $H$ is ample. Then $K_{\mcF'}+\Delta'$ is semiample by Corollary \ref{gmm}.\\

We claim that $M_{X'}^{\Delta'+G'} \sim K_{\mcF'}+\Delta'$. Indeed, as observed above, the induced morphism $f':X' \to Y$ is equidimensional and $(\mcF',B')$ has property ($*$). By \cite[Prop 3.6]{ACSS}, we have $M_{X'}^{B'+G'} \sim K_{\mcF'}+B'$. Now, by definition, 
\begin{center}
$M_{X'}^{\Delta'+G'}= K_{X'}+B'+G'+A'-f'^*(K_{Y'}+B_{Y'}^{B'+G'+A'})$. 
\end{center}
For $P\subset Y$ a prime divisor, we have $(X', B'+G'+a_P^{B'+G'}f'^*P)$ is lc over $\eta_P$ and has some lc center. Now, $\phi$ is a $K_X+B+G+a_P^{B'+G'}f^*P$-MMP over $\eta_P$. Thus any lcc of $(X', B'+G'+a_P^{B'+G'}f'^*P)$ over $\eta_P$ arises as image of some lcc of $(X, B+G+a_P^{B'+G'}f^*P)$ over $\eta_P$ (note the last entity is lc by property ($*$)). $A$ being a general ample divisor, does not contain any such lcc for any such $P$. Thus, neither does $A'$. We conclude that $(X', B'+G'+A'+a_P^{B'+G'}f'^*P)$ is lc over $\eta_P$ and hence $a_P^{\Delta'+G'}\geq a_P^{B'+G'}$ for all prime divisors $P\subset Y$. It follows that $B_{Y'}^{\Delta'+G'}=B_{Y'}^{B'+G'}$. From this, we can infer that $M_{X'}^{\Delta'+G'}= M_{X'}^{B'+G'}+A' \sim K_{\mcF'}+\Delta'$. This proves the claim. We also note that $M_X^{\Delta+G} =M_X^{B+G}+A \sim K_\mcF+\Delta$. \\

Now $M_{X'}^{\overline{\Delta}'}=M_{X'}^{\Delta'+G'} \sim K_{\mcF'}+\Delta'$ is semiample. Let $X \xleftarrow{p}W \xrightarrow{q}X'$ be a resolution of indeterminacy of $\phi:X \dashrightarrow X'$. Then $p^*(K_\mcF+\Delta)=q^*(K_{\mcF'}+\Delta') +F$ for some $F \geq 0$ which is exceptional over $X'$. Now $K_\mcF+\Delta \sim K_X+\Delta+G-f^*(K_Y+B_Y^{\Delta+G})= K_X+\Delta-f^*(K_Y+B_Y^{\Delta})+E= M_X^\Delta +E$. It follows that 
\begin{equation} \label{ }
p^*M_X^\Delta = q^*M_{X'}^{\Delta'+G'}+F-p^*E.
\end{equation}
By \cite[Prop 2.21]{ACSS}, $F-p^*E \sim D \geq 0$. Now we have $F-p^*E \equiv_{X'}p^*M_X^\Delta$. Since $M_X^\Delta \equiv _Y 0$, it follows that $F-p^*E \equiv_{X'}0$. Since $F$ and $p^*E$ are both exceptional over $X'$, $p^*E \geq F$ by negativity lemma over $X'$. Combining this with the above observation, it follows that $p^*E=F$. By equation \ref{ }, it follows that $p^*M_X^\Delta=q^*M_{X'}^{\Delta'+G'}$ is semiample.

\end{proof}

\end{corollary}




\bibliographystyle{hep}
\bibliography{references}

\end{document}